\documentclass{amsart}

\usepackage{amscd,amsmath,xypic,amssymb,combelow,tikz-cd,etoolbox,calligra,mathrsfs,enumitem,mathtools,epigraph}

\usepackage[russian,english]{babel}

\makeatletter
\patchcmd{\@settitle}{\uppercasenonmath\@title}{}{}{}
\makeatother

\xyoption{all}
\CompileMatrices

\makeatletter
\@addtoreset{equation}{section}
\makeatother

\newtheorem{theorem}[subsection]{Theorem}

\newtheorem{proposition}[subsection]{Proposition}
\newtheorem{lemma}[subsection]{Lemma}
\newtheorem{corollary}[subsection]{Corollary}

\newtheorem{definition}[subsection]{Definition}

\newtheorem{claim}[subsection]{Claim}
\newtheorem{example}[subsection]{Example}
\newtheorem{remark}[subsection]{Remark}

\def\loccitt{\emph{loc. cit.}}
\def\loccit{\emph{loc. cit. }}

\def\fW{{\mathfrak{W}}}

\def\fZ{{\mathfrak{Z}}}

\def\BA{{\mathbb{A}}}
\def\BC{{\mathbb{C}}}
\def\BK{{\mathbb{K}}}

\def\BN{{\mathbb{N}}}
\def\BF{{\mathbb{F}}}
\def\BP{{\mathbb{P}}}

\def\BQ{{\mathbb{Q}}}
\def\BZ{{\mathbb{Z}}}

\def\CA{{\mathcal{A}}}
\def\CB{{\mathcal{B}}}
\def\DD{{\mathcal{D}}}

\def\CL{{\mathcal{L}}}

\def\CR{{\mathcal{R}}}
\def\CS{{\mathcal{S}}}

\def\CV{{\mathcal{V}}}

\def\Hom{\textrm{Hom}}

\def\vs{\varsigma}

\def\and{\textrm{ }\&\textrm{ }}

\def\sym{\textrm{sym}}
\def\Sym{\textrm{Sym}}

\def\oCS{\mathring{\CS}}

\def\nn{{{\BN}}^I}

\def\bs{{\boldsymbol{\vs}}}

\def\bn{\boldsymbol{n}}
\def\bvs{\boldsymbol{\varsigma}}

\def\bphi{\boldsymbol{\phi}}

\def\op{\text{op}}

\def\oii{\overrightarrow{ii}}
\def\oij{\overrightarrow{ij}}
\def\oji{\overrightarrow{ji}}
\def\loc{\text{loc}}
\def\eloc{\emph{loc}}

\begin{document}

\title[Shuffle algebras for quivers and wheel conditions]{\Large{\textbf{Shuffle algebras for quivers and wheel conditions}}}

\author[Andrei Negu\cb t]{Andrei Negu\cb t}
\address{MIT, Department of Mathematics, Cambridge, MA, USA}
\address{Simion Stoilow Institute of Mathematics, Bucharest, Romania}
\email{andrei.negut@gmail.com}

\maketitle

\begin{abstract} We show that the shuffle algebra associated to a doubled quiver (determined by 3-variable wheel conditions) is generated by elements of minimal degree. Together with results of \cite{VV} and \cite{Zhao}, this implies that the aforementioned shuffle algebra is isomorphic to the localized $K$-theoretic Hall algebra associated to the quiver by \cite{G, SV Hilb, YZ 0}. With small modifications, our theorems also hold under certain specializations of the equivariant parameters, which will allow us in \cite{Quiver 3} to give a generators-and-relations description of the Hall algebra of any curve over a finite field (which is a shuffle algebra due to \cite{KSV}). When the quiver has no edge loops or multiple edges, we show that the shuffle algebra, localized $K$-theoretic Hall algebra, and the positive half of the corresponding quantum loop group are all isomorphic; we also obtain the non-degeneracy of the Hopf pairing on the latter quantum loop group.

\end{abstract}

$$$$

\epigraph{\emph{\begin{otherlanguage*}{russian}Где-то есть люди, для которых теорема верна \end{otherlanguage*}}}

\section{Introduction}

\medskip

\subsection{} A quiver is a finite oriented graph. Fix a quiver $Q$ with vertex set $I$ and edge set $E$; edge loops and multiple edges are allowed. For any $\bn = (n_i)_{i \in I} \in \nn$ \footnote{Throughout the present paper, the set $\BN$ is considered to include 0}, an $\bn$-dimensional quiver representation is a point in the vector space:
$$
Z_{\bn} = \bigoplus_{\oij \in E} \text{Hom}(\BC^{n_i}, \BC^{n_j})
$$
i.e. a collection of linear maps indexed by the edges of the quiver. As one is often interested in quiver representations only up to automorphism, it is meaningful to consider the conjugation action of the following group on $Z_{\bn}$:
$$
G_{\bn} = \prod_{i \in I} GL_{n_i}(\BC)
$$
Thus, a quiver representation up to automorphism is a point in the stack:
$$
\fZ_{\bn} = Z_{\bn} / G_{\bn}
$$
The geometry of this stack is the source of many beautiful constructions in geometric representation theory: for example, Lusztig's categorification of the positive half of quantum groups using constructible sheaves on $\fZ_{\bn}$. Using the microsupport construction, one can change perspective and consider instead the $K$-theoretic Hall algebra studied by Schiffmann-Vasserot (\cite{G, Lu, SV Hilb, YZ 0}, see \cite{S Lect} for an overview and historical perspective):
\begin{equation}
\label{eqn:k-ha intro}
K = \bigoplus_{\bn \in \nn} K_T(T^*\fZ_{\bn})
\end{equation}
A fuller description of $T^*\fZ_{\bn}$, as well as the action of the following torus on it: 
$$
T = \BC^* \times \prod_{e \in E} \BC^*
$$
will be recalled in Section \ref{sec:k-ha}. We will first describe our results for the $K$-theoretic Hall algebra equivariant with respect to the full torus $T$, and then describe the modifications necessary for treating the case of subtori in Subsection \ref{sub:intro special}. As shown in \cite{SV Hilb}, $K$ is a $\nn$-graded algebra over $\text{Rep}_T = \BZ[q^{\pm 1}, t_e^{\pm 1}]_{e \in E}$. If we let $\BF$ denote the fraction field of $\text{Rep}_T$, we may define the localized $K$-theoretic Hall algebra as:
\begin{equation}
\label{eqn:intro k-ha loc}
K_{\text{loc}} = K \bigotimes_{\text{Rep}_T} \BF
\end{equation}
There is a natural map (\cite{S Lect}, we will recall the construction in Subsection \ref{sub:kha to shuffle}):
\begin{equation}
\label{eqn:shuf intro}
K_{\text{loc}} \rightarrow \CV = \bigoplus_{\bn = (n_i)_{i \in I} \in \nn} \BF[z_{i1}^{\pm 1}, \dots, z_{in_i}^{\pm 1}]^{\sym}_{i \in I}
\end{equation}
which is an algebra homomorphism, where the multiplication in $K_{\text{loc}}$ is the convolution product \eqref{eqn:conv prod}, and the multiplication in $\CV$ is the shuffle product \eqref{eqn:shuf prod}. Yu Zhao (\cite{Zhao}) showed that the map above actually lands in the subalgebra of $\CV$ consisting of Laurent polynomials which satisfy the 3-variable wheel conditions:
\begin{equation}
\label{eqn:iota intro}
K_{\text{loc}} \xrightarrow{\iota} \CS = \left\{ R \in \CV \text{ s.t. } R \Big|_{z_{ia} = \frac {qz_{jb}}{t_e} = q z_{ic}} = R \Big|_{z_{ja} = t_e z_{ib} = q z_{jc}} = 0 \right\}
\end{equation}
for any edge $e = \oij$ and all $a \neq c$ (and further $a \neq b \neq c$ if $i = j$). The vanishing properties of $R \in \CV$ above are inspired by those of \cite{FHHSY, FO} for quantum groups, hence we refer to them as ``wheel conditions". Let us consider the spherical subalgebra:
$$
\oCS \subseteq \CS
$$
generated by Laurent polynomials in one variable (i.e. corresponding to the direct summands $\bn = \bs_i := (\underbrace{\dots,0,1,0,\dots}_{1 \text{ on }i\text{-th spot}})$ of $\CV$ in \eqref{eqn:shuf intro}, $\forall i \in I$). Our main result is: \\

\begin{theorem}
\label{thm:intro} 

(\emph{Theorem \ref{thm:main}}) We have $\oCS = \CS$. \\

\end{theorem}

\noindent Since $\iota$ is an isomorphism in degrees $\bs_i$ for all $i \in I$, this implies that $\iota$ is surjective. As $\iota$ was shown to be injective by Varagnolo-Vasserot in \cite{VV}, we conclude that:
\begin{equation}
\label{eqn:iso intro}
K_\loc \cong \CS
\end{equation}
Thus the shuffle algebra provides an explicit model for the localized K-HA. The advantage of such a result is that one can construct numerous elements of the K-HA just by producing symmetric Laurent polynomials which satisfy the wheel conditions (this approach was used in \cite{Shuf,K,Hecke} to categorify the celebrated Heisenberg algebra action on the cohomology of Hilbert schemes, of Nakajima and Grojnowski). \\

\noindent In \cite{Quiver 2}, we will use the constructions in the present paper to consider shuffle algebras as quasi-triangular Hopf algebras, and explore the connections between their universal $R$-matrices and the geometric constructions of \cite{AO, MO, O1, O2, OS}. This will lead us to a conjectural realization of the Kac polynomial of the quiver (at $t=1$) as the graded dimension of a certain explicit ``slope subalgebra" of $\CS$. \\

\noindent Our proof of Theorem \ref{thm:intro} relies on the combinatorics of ``loop words" introduced in \cite{NT}, a treatment which was inspired by \cite{LR, L, R}. In fact, our proof of Theorem \ref{thm:intro} also gives another proof of \cite[Theorem 1.8]{NT} in the simply laced case. However, while \loccit heavily used particular features of quantum loop groups, our proof of Theorem \ref{thm:intro} is direct and could in principle be applied to numerous other flavors of shuffle algebras. For example, in \cite{Cartan} we adapt the techniques developed herein to the setting of quantum loop groups for arbitrary symmetric Cartan matrices, while in \cite{Arbitrary} we study the more general setting of $K$-theoretic Hall algebras associated to quivers (see \cite{KS} for the original construction, and \cite{Pad 2} for a new viewpoint). \\

\subsection{} 
\label{sub:intro special} 

We will now address the analogue of the results above when $\{q,t_e\}_{e \in E}$ are no longer independent formal symbols, but non-zero elements of a field $\BK$ of characteristic 0. In this generality, define the shuffle algebra as the $\BK$-vector subspace:
\begin{equation}
\label{eqn:shuffle spec intro}
_{\BK}\CS \subset \bigoplus_{\bn = (n_i)_{i \in I} \in \nn} \BK[z_{i1}^{\pm 1}, \dots, z_{in_i}^{\pm 1}]^{\sym}_{i \in I}
\end{equation}
consisting of symmetric Laurent polynomials which satisfy \eqref{eqn:divisible 2}. If we let $_{\BK}\oCS \subseteq {_{\BK}\CS}$ denote the subalgebra generated by Laurent polynomials in one variable, we have: \\

\begin{theorem}
\label{thm:intro spec} 

(\emph{Corollary \ref{cor:gen}}) We have $_{\BK}\oCS = {_{\BK}\CS}$, under \textbf{Assumption \begin{otherlanguage*}{russian}Ъ\end{otherlanguage*}}, namely:
\begin{equation}
\label{eqn:intro assumption}
\exists \text{ a field homomorphism } \rho : \BK \rightarrow \BC \text{ s.t. } |\rho(q)| < |\rho(t_e)| < 1, \ \forall e \in E
\end{equation}

\end{theorem}

\medskip

\noindent Let us now give three main applications of Theorem \ref{thm:intro spec}. The first involves setting:
$$
\BK = \text{Frac}(\text{Rep}_H)
$$
for a subtorus $H \subset T$ (and letting $q,t_e \in \text{Rep}_H$ be the characters of $H$ restricted from the homonymous characters of $T$). In this context, we will prove in Proposition \ref{prop:special yu} that there exists an algebra homomorphism analogous to \eqref{eqn:iota intro}:
\begin{equation}
\label{eqn:special iota}
K_{H,\text{loc}} \xrightarrow{\iota_H} {_{\BK}\CS}
\end{equation}
where the left-hand side is defined as in \eqref{eqn:k-ha intro} and \eqref{eqn:intro k-ha loc}, but with $T, \BF$ replaced by $H,\BK$. If the characters $\{q,t_e\}_{e \in E}$ of $H$ satisfy Assumption \begin{otherlanguage*}{russian}Ъ\end{otherlanguage*} of \eqref{eqn:intro assumption}, then Theorem \ref{thm:intro spec} implies that the map \eqref{eqn:special iota} is surjective (see Corollary \ref{cor:special surj}). We emphasize the fact that this holds for either deformed or non-deformed $K$-theoretic Hall algebras, in the language of \cite{VV}, as long as Assumption \begin{otherlanguage*}{russian}Ъ\end{otherlanguage*} holds. Whenever $\iota_H$ is also injective (in Subsection \ref{sub:special k-ha} we will recall a quite general criterion, due to \cite{VV}, for when this happens) then we conclude that $\iota_H$ is an isomorphism and thus $K_{H,\text{loc}}$ is generated by elements of minimal degree (see Corollary \ref{cor:special iso} for details). \\

\noindent The second main application of Theorem \ref{thm:intro spec} is when $\BK = \BC$ and the complex numbers $\{t_1, \dots, t_g, q/t_1,\dots, q/t_g\}$ are set equal to the inverses of the Weil numbers of a smooth projective genus $g$ curve $X$ over $\BF_{q^{-1}}$. By \cite{KSV} and \cite{SV}, the spherical Hall algebra of the category of coherent sheaves on $X$ is isomorphic to the shuffle algebra $_{\BK}\oCS$ for the quiver with one vertex and $g$ loops. Assumption \begin{otherlanguage*}{russian}Ъ\end{otherlanguage*} is verified because the Weil numbers of $X$ have absolute value $q^{-\frac 12}$, according to the Riemann hypothesis. Then Theorem \ref{thm:intro spec} implies that the spherical Hall algebra in question is isomorphic to $_{\BK}\CS$, and thus can be described by the explicit wheel conditions \eqref{eqn:divisible 2}. In \cite{Quiver 3}, we will develop this argument further to obtain a generators-and-relations description of the spherical Hall algebra in question, as well as extend the result to the subalgebra generated by absolutely cuspidal eigenforms of the whole Hall algebra of $X$. \\

\subsection{}
\label{sub:intro km}

For the third main application of Theorem \ref{thm:intro spec}, assume that $Q$ has no edge loops or multiple edges. Consider the following chain of algebra homomorphisms: 
\begin{equation}
\label{eqn:triple composition}
\Upsilon : U_q^+(L \mathfrak{g}_Q) \longrightarrow K_{\BC^*,\text{loc}} \xrightarrow{\iota_{\BC^*}} {_q\CS'}
\end{equation}
where: \\

\begin{itemize}[leftmargin=*]

\item the left-most algebra is the positive half of the quantum loop group associated to the (Cartan matrix of the) quiver $Q$ (see Definition \ref{def:quantum group}); \\

\item the middle algebra is the localized $K$-theoretic Hall algebra, equivariant with respect to the one-dimensional torus: 
$$
\BC^* \hookrightarrow \BC^* \times \prod_{e \in E} \BC^*, \qquad a \mapsto (a^2,a,\dots,a)
$$
The standard character of the one-dimensional torus above is denoted by $q^{\frac 12}$; \\

\item the right-most algebra is the shuffle algebra of Subsection \ref{sub:r}, with the ground field changed to $\BK = \BQ(q^{\frac 12})$ and $t_e$ set equal to $q^{\frac 12}$ for all edges $e$ (note that the wheel conditions \eqref{eqn:wheel} and \eqref{eqn:divisible 2} are equivalent in the case at hand). \\

\end{itemize}

\noindent The first arrow in \eqref{eqn:triple composition} is one of the main motivations for $K$-theoretic Hall algebras, as they provide models for the positive halves of quantum loop groups. Assumption \begin{otherlanguage*}{russian}Ъ\end{otherlanguage*} applies to the situation in the third bullet above, so the composition of the maps in \eqref{eqn:triple composition} is surjective. Together with results of \cite{Quiver 3}, we obtain the following. \\

\begin{theorem}
\label{thm:triple iso}

For any quiver $Q$ without edge loops or multiple edges, the map $\Upsilon$ of \eqref{eqn:triple composition} is an isomorphism. Since the first arrow in \eqref{eqn:triple composition} was shown to be surjective in \cite[Theorem A]{VV}, we conclude that all three algebras in \eqref{eqn:triple composition} are isomorphic. \\

\end{theorem}

\noindent Theorem \ref{thm:triple iso} was proved for cyclic quivers in \cite{Cyc} and for finite type quivers in \cite{NT}. For more general quivers, we will tackle the analogue of Theorem \ref{thm:triple iso} in \cite{Cartan, Arbitrary}. \\

\begin{corollary}
\label{cor:hopf pairing}

For any quiver $Q$ without edge loops or multiple edges, the usual ``bialgebra" pairing (see \eqref{eqn:pairing loop} for details):
$$
U_q^+(L \mathfrak{g}_Q) \otimes U_q^-(L \mathfrak{g}_Q) \xrightarrow{\langle \cdot, \cdot \rangle} \BQ(q^{\frac 12})
$$
is non-degenerate. \\

\end{corollary}

\subsection{} The structure of the present paper is the following. \\

\begin{itemize}[leftmargin=*]

\item In Section \ref{sec:k-ha}, we recall known facts about the $K$-theoretic Hall algebra associated to the quiver $Q$, and about its relationship to the shuffle algebra $\CS$. \\

\item In Section \ref{sec:proof}, we introduce certain features of words in letters $i^{(d)}$ (where $i \in I$ and $d \in \BZ$), and use these features to prove Theorem \ref{thm:intro}. \\

\item In Section \ref{sec:r}, we explain how to adapt our proof to account for ``twists" of the shuffle product on $\CS$, including a particular choice which yields important examples: quantum loop groups associated to quivers, and Hall algebras of curves over finite fields. This particular twist admits an important Hopf algebra structure.  \\ 

\item In Section \ref{sec:tori}, we explain how to adapt the contents of the present paper to the situation when the parameters $\{q,t_e\}_{e\in E}$ are not generic, but are non-zero elements of a certain field $\BK$. We prove the theorems and corollaries from Subsections \ref{sub:intro special} and \ref{sub:intro km}. \\

\end{itemize}

\noindent I would like to thank Boris Feigin, Francesco Sala, Olivier Schiffmann, Tudor Pădurariu, Alexander Tsymbaliuk (with special thanks for the substantial feedback), Michela Varagnolo, \'Eric Vasserot and Yu Zhao for many interesting conversations about $K$-theoretic Hall algebras, shuffle algebras and much more. I gratefully acknowledge NSF grants DMS-$1760264$ and DMS-$1845034$, as well as support from the Alfred P.\ Sloan Foundation and the MIT Research Support Committee.\\

\section{$K$-theoretic Hall algebras of (doubled) quivers} 
\label{sec:k-ha}

\medskip

\subsection{} A quiver is a finite oriented graph $Q$ with vertex set $I$ and edge set $E$; edge loops and multiple edges are allowed. Given a collection of non-negative integers $\bn = (n_i)_{i \in I}$, a representation of $Q$ of dimension $\bn$ is a collection of linear maps:
$$
\bphi = \Big( \phi_e : \BC^{n_i} \rightarrow \BC^{n_j} \Big)_{\forall e = \oij \in E}
$$
(if there are several edges between two given vertices $i$ and $j$, then there will be several linear maps $\phi_e$ between $\BC^{n_i}$ and $\BC^{n_j}$ part of the datum $\bphi$). The space of representations of the quiver is simply the affine space parametrizing all such $\bphi$'s:
$$
Z_{\bn} = \bigoplus_{e = \oij \in E} \Hom(\BC^{n_i}, \BC^{n_j})
$$
One is often interested in studying representations up to isomorphism, i.e. the orbits of $Z_{\bn}$ under the action of the group:
$$
G_{\bn} = \prod_{i \in I} GL_{n_i}(\BC)
$$
where $(g_i)_{i \in I}$ sends $(\phi_e)_{e = \oij}$ to $(g_j \phi_e g_i^{-1})_{e = \oij}$. The corresponding quotient:
$$
\fZ_{\bn} = Z_{\bn} / G_{\bn}
$$
is the stack of $\bn$-dimensional quiver representations modulo isomorphism. A lot of beautiful mathematics seeks to understand the enumerative properties of the stack $\fZ_{\bn}$, such as counting the number of its points when $\BC$ is replaced by a finite field. \\

\noindent A more recent point of view is to consider other invariants of the stack $\fZ_{\bn}$, such as its equivariant (co)homology and $K$-theory. Though there are several points of view in this direction, we will focus on the one developed by Schiffmann, Varagnolo, Vasserot and others: to study the algebras that arise from $K$-theory groups of cotangent representation stacks, in a way which is to Nakajima's  construction of quantum group representations on the $K$-theory groups of double quiver varieties (see \cite{Nak}) as algebras are to modules. Specifically, one takes the $K$-theory groups \footnote{In the present paper, $K$ will denote the Grothendieck group of all coherent sheaves on a certain variety or stack; the usual notation for this concept in the literature is either $K_0$ or $G$.}:
\begin{equation}
\label{eqn:k-ha}
K_{\bn} = K_{T}(T^*\fZ_{\bn}) = K_{T \times G_{\bn}}(\mu_{\bn}^{-1}(0))
\end{equation}
Above, the cotangent stack $T^*\fZ_{\bn} = \mu_{\bn}^{-1}(0) / G_{\bn}$ is defined via the moment map:
\begin{equation}
\label{eqn:moment map}
\mu_{\bn} : T^*Z_{\bn} \rightarrow \bigoplus_{i \in I} \text{End}(\BC^{n_i})
\end{equation}
given by (recall that for a vector space $V$, we have $T^*V \cong V \times V^*$):
$$
\mu_{\bn} \Big( \phi_e : \BC^{n_i} \rightarrow \BC^{n_j}, \phi^*_e : \BC^{n_j} \rightarrow \BC^{n_i} \Big)_{\forall e = \oij \in E} = \sum_{e \in E} (\underbrace{\phi_e \phi^*_e}_{\in \text{End}(\BC^{n_j})} - \underbrace{\phi^*_e \phi_e}_{\in \text{End}(\BC^{n_i})} ) 
$$
In \eqref{eqn:k-ha}, one considers equivariant $K$-theory with respect to the torus: 
\begin{equation}
\label{eqn:torus}
T = \BC^* \times \prod_{e \in E} \BC^*
\end{equation}
which acts on $T^*Z_{\bn}$ via:
\begin{equation}
\label{eqn:torus action}
(\bar{q},\bar{t}_e)_{e \in E} \cdot (\phi_e, \phi^*_e)_{e \in E} = \left( \frac 1{\bar{t}_e} \phi_e, \frac {\bar{t}_e}{\bar{q}} \phi^*_e \right)_{e \in E}
\end{equation}
Thus, $K_{\bn}$ is a module over $K_T(\text{point}) = \text{Rep}_T = \BZ[q^{\pm 1}, t^{\pm 1}_e]_{e \in E}$, where $q$ and $t_e$ denote the natural dual coordinates on the factors of the product \eqref{eqn:torus}. Note that one needs the torus weights of $\phi_e$ and $\phi_e^*$ to multiply to one and the same weight for all edges $e \in E$, namely $q^{-1}$, in order for the map \eqref{eqn:moment map} to be $T$-equivariant. \\

\begin{remark}
\label{rem:1 loop}

When $Q$ is the Jordan quiver (namely one vertex and one loop), for any $n \in \BN$, the stack $T^*\fZ_n$ can be identified with the commuting stack:
$$
\emph{Comm}_n/GL_n(\BC) = \Big\{X,Y \in \emph{Mat}_{n \times n}(\BC), [X,Y] = 0 \Big\}/\text{simultaneous conjugation}
$$
This was one of the main examples that spurred the study of K-HA's in the present context (see \cite{SV Hilb}, which also provides a connection to Hilbert schemes of points). \\

\end{remark}

\subsection{} So far, each $K_{\bn}$ is a $\text{Rep}_T$-module. To construct an algebra (the so-called $K$-theoretic Hall algebra of the quiver $Q$), we need to sum over all dimension vectors:
$$
K = \bigoplus_{\bn \in \nn} K_{\bn}
$$
The multiplication in $K$ is given by the following convolution product (\cite{S Lect}). For any dimension vectors $\bn, \bn' \in \nn$, one considers the following stack of extensions:
\begin{equation}
\label{eqn:extension diagram}
\xymatrix{& \fW_{\bn,\bn'} \ar@{->>}[ld]_{p_1} \ar@{^{(}->}[rd]^{p_2} & \\
T^* \fZ_{\bn} \times T^*\fZ_{\bn'} & & T^*\fZ_{\bn+\bn'}}
\end{equation}
where the map $p_2$ is the embedding of the closed subset of quiver representations:
\begin{equation}
\label{eqn:collection}
\Big( \phi_e : \BC^{n_i+n_i'} \rightarrow \BC^{n_j+n_j'}, \phi^*_e : \BC^{n_j+n_j'} \rightarrow \BC^{n_i+n_i'} \Big)_{\forall e = \oij \in E} \in T^*\fZ_{\bn+\bn'}
\end{equation}
which preserve a fixed collection of subspaces $\BC^{n_i} \hookrightarrow \BC^{n_i+n_i'}$, for all $i \in I$. The map $p_1$ sends the collection \eqref{eqn:collection} to the pair comprising of the restricted linear maps to the subspaces $\BC^{n_i}$ (which yields a point of $T^*\fZ_{\bn}$) and the induced linear maps on the quotients $\BC^{n'_i} := \BC^{n_i+n_i'} / \BC^{n_i}$ (which yields a point of $T^*\fZ_{\bn'}$). Also define:
\begin{equation}
\label{eqn:line bundle}
[\CL_{\bn,\bn'}] = \prod_{i \in I} \left[ \left( - \frac 1q \right)^{n_in_i'} \frac {(\det V_i)^{\otimes n_i'}}{(\det V_i')^{\otimes n_i}} \right]
\end{equation}
as $\pm$ the class of an equivariant line bundle on $\fW_{\bn,\bn'}$ (above, $V_i$ and $V_i'$ denote the pull-backs via $p_1$ to $\fW_{\bn,\bn'}$ of the tautological rank $n_i$ and $n_i'$ vector bundles on the stacks $T^*\fZ_{\bn}$ and $T^*\fZ_{\bn'}$, respectively). With this in mind, the operation:
\begin{equation}
\label{eqn:conv prod}
K_{\bn} \otimes K_{\bn'} \xrightarrow{*} K_{\bn+\bn'}, \qquad \alpha * \alpha' = p_{2*} \Big([\CL_{\bn,\bn'}] \cdot p_1^! (\alpha \boxtimes \alpha') \Big)
\end{equation}
gives rise to an associative $\text{Rep}_T$-algebra structure on $K$ (the pull-back $p_1^!$ is the refined Gysin map, see \cite{Zhao 0} for an introduction). \\

\begin{remark}
\label{rem:g loops}

When $Q$ is the quiver with one vertex and $g$ loops, Schiffmann-Vasserot studied the notions above in detail in \cite{SV}, and connected $K$ to the Hall algebra of the category of coherent sheaves on a genus $g$ smooth curve over the finite field $\BF_{q^{-1}}$ (in this context, the equivariant parameters $t_1,\dots,t_g, qt_1^{-1},\dots,qt_g^{-1}$ must be specialized to the inverses of the Weil numbers of the curve). \\ 

\end{remark}

\begin{remark} 
\label{rem:modify}

A feature (or bug, depending on one's point of view) of $K$-theory, which does not arise in cohomology, is that one could in principle replace \eqref{eqn:line bundle} by other line bundles satisfying an appropriate compatibility condition. If one were to perform such a replacement, one would also need to multiply the functions \eqref{eqn:def zeta} by an appropriate monomial (see Section \ref{sec:r} for a particularly important example). \\

\end{remark}

\subsection{} 
\label{sub:kha to shuffle} 

The closed embedding $i : \mu_{\bn}^{-1}(0) \hookrightarrow T^*Z_{\bn}$ induces a map:
\begin{equation}
\label{eqn:iota 0}
\iota : K_{\bn} \xrightarrow{i_*} K_{T \times G_{\bn}}(T^*Z_n) \cong K_{T \times G_{\bn}}(\text{point}) \cong \text{Rep}_T[\dots,z^{\pm 1}_{ia},\dots]^{\sym}_{i \in I, 1 \leq a\leq n_i}
\end{equation}
The first isomorphism is the restriction map from a vector space to the origin, while the second isomorphism is a restatement of the fact that the representation ring of $GL_n(\BC)$ is the ring of symmetric Laurent polynomials in $n$ variables (indeed, the word ``sym" refers to Laurent polynomials which are symmetric in the variables $z_{i1},\dots,z_{in_i}$ for each $i$ separately). In \eqref{eqn:iota 0}, we write $z_{ia}$ for the character of the standard maximal torus of $GL_{\bn}$ which is ``dual" to the one-parameter subgroup:
$$
\text{diag}(\underbrace{1,\dots,1,t,1,\dots,1}_{t\text{ on }a\text{-th position}} ) \hookrightarrow GL_{n_i} \hookrightarrow GL_{\bn}
$$
Taking the direct sum of \eqref{eqn:iota 0} over all dimension vectors $\bn \in \nn$, we obtain a map:
\begin{equation}
\label{eqn:iota}
K \xrightarrow{\iota} \CV_{\text{int}} = \bigoplus_{\bn \in \nn} \text{Rep}_T[z^{\pm 1}_{i1},\dots,z^{\pm 1}_{in_i}]_{i \in I}^{\sym}
\end{equation}
The notation ``int" stands for ``integral", since the coefficient ring of $\CV_{\text{int}}$ is $\text{Rep}_T$. Let us endow $\CV_{\text{int}}$ with the following shuffle product, following \cite{SV} (see also \cite{E,FHHSY, FO, FT} for other representation-theoretic incarnations of this shuffle product):
\begin{equation}
\label{eqn:shuf prod}
R(\dots,z_{i1},\dots,z_{in_i},\dots) * R'(\dots,z_{i1},\dots,z_{in'_i},\dots) = 
\end{equation}
$$
\text{Sym} \left[\frac {R(\dots,z_{i1},\dots,z_{in_i},\dots)R'(\dots,z_{i,n_i+1},\dots,z_{i,n_i+n_i'},\dots)}{\prod_{i\in I} n_i! \prod_{i\in I} n_i'!}\mathop{\prod^{i,j \in I}_{1 \leq a \leq n_i}}_{n_{j} < b \leq n_{j}+n'_{j}} \zeta_{ij} \left( \frac {z_{ia}}{z_{jb}} \right) \right]
$$
where ``Sym" denotes symmetrization with respect to the variables $z_{i1},\dots,z_{i,n_i+n_i'}$ for each $i \in I$ separately, and for any $i,j \in I$ we define the following function:
\begin{equation}
\label{eqn:def zeta}
\zeta_{ij}(x) = \left(\frac {1-xq^{-1}}{1-x} \right)^{\delta_j^i} \prod_{e = \oij \in E} (1-t_e x) \prod_{e = \oji \in E} \left(1 - \frac {qx}{t_e} \right)
\end{equation}
Note that even though the right-hand side of \eqref{eqn:shuf prod} seemingly has simple poles at $z_{ia} - z_{ib}$ for all $i \in I$ and all $a < b$, these poles disappear in $\text{Sym}[\dots]$, as the orders of such poles in a symmetric rational function must be even. Therefore, the shuffle product \eqref{eqn:shuf prod} preserves the direct sum of Laurent polynomial rings $\CV_{\text{int}}$ of \eqref{eqn:iota}. The specific formula in \eqref{eqn:def zeta} was motivated by the following result. \\

\begin{proposition}
\label{prop:homomorphism}

The map \eqref{eqn:iota} is an algebra homomorphism. \\

\end{proposition}

\begin{proof} The proof is standard, so we encourage the interested reader to go through the details (we will only sketch the main steps, and leave the straightforward details as an exercise). Consider the following commutative diagram, extending \eqref{eqn:extension diagram}:

$$
\xymatrix{& W_{\bn,\bn'} \ar@{->>}[ld]_{p_1} \ar@{^{(}->}[rd]^{p_2} \ar@{^{(}->}[dd]^{i_2} & \\
\mu_{\bn}^{-1}(0) \times \mu_{\bn'}^{-1}(0) \ar@{^{(}->}[dd]^{i_1} & & \mu_{\bn+\bn'}^{-1}(0) \ar@{^{(}->}[dd]^{i_3} \\ 
& Y_{\bn,\bn'} \ar@{->>}[ld]_{\tilde{p}_1} \ar@{^{(}->}[rd]^{\tilde{p}_2} \ar[dd]^{\pi_2} & \\
T^*Z_{\bn} \times T^*Z_{\bn'} \ar[dd]^{\pi_1} & & T^*Z_{\bn+\bn'} \ar[dd]^{\pi_3} \\
& X_{\bn,\bn'} \ar@{->>}[ld]_{\bar{p}_1} \ar@{^{(}->}[rd]^{\bar{p}_2} & \\
\oplus_{i \in I} \text{End}(\BC^{n_i}) \times \oplus_{i \in I} \text{End}(\BC^{n_i'}) & & \oplus_{i \in I} \text{End}(\BC^{n_i+n_i'})}
$$
where the spaces in the middle column are defined so that $p_2,\tilde{p}_2,\bar{p}_2$ are the closed embeddings of the loci of linear maps between the collection of vector spaces $\BC^{n_i+n_i'}$ which preserve the collection of vector subspaces $\BC^{n_i}$. Then we have:
$$
\iota(\alpha * \alpha') = i_{3*} \Big[ p_{2*} \Big([\CL_{\bn,\bn'}] \cdot p_1^!(\alpha \boxtimes \alpha') \Big) \Big] \Big|_\circ =  \tilde{p}_{2*} \Big[ i_{2*} \Big([\CL_{\bn,\bn'}] \cdot p_1^!(\alpha \boxtimes \alpha') \Big) \Big] \Big|_\circ 
$$
where $|_\circ$ denotes restriction to the origin of an affine space. By the excess intersection formula, the right-hand side of the expression above equals:
\begin{equation}
\label{eqn:above}
\tilde{p}_{2*} \Big[ \tilde{p}_{1}^! \Big( [\wedge^\bullet A^\vee] \cdot [\CL_{\bn,\bn'}] \cdot i_{1*}(\alpha \boxtimes \alpha') \Big) \Big] \Big|_\circ 
\end{equation}
where $A = \text{Ker } \bar{p}_1$. The analogous formula allows us to replace $\tilde{p}_{2*}(\tilde{p}_1^!(\dots))|_\circ$ in the formula above by $[\wedge^\bullet B^\vee] \cdot (\dots)|_\circ$, where $B = \text{Coker } \tilde{p}_2$, so \eqref{eqn:above} equals:
$$
[\wedge^\bullet A^\vee] \cdot [\wedge^\bullet B^\vee] \cdot [\CL_{\bn,\bn'}] \cdot i_{1*}(\alpha \boxtimes \alpha') \Big|_\circ = [\wedge^\bullet A^\vee] \cdot [\wedge^\bullet B^\vee] \cdot [\CL_{\bn,\bn'}] \cdot \Big(\iota(\alpha) \boxtimes \iota(\alpha') \Big)
$$
Explicitly, the $T \times G_{\bn} \times G_{\bn'}$ characters of $A$, $B$, $\CL_{\bn,\bn'}$ are:
$$
A = \bigoplus_{i \in I}\Hom(\BC^{n_i'}, \BC^{n_i}) \qquad \quad \ \ \Rightarrow \ \ \quad \qquad \chi_{T \times G_{\bn} \times G_{\bn'}}(A) = \mathop{\sum^{i\in I}_{1\leq a \leq n_i}}_{n_i < b \leq n_i+n_i'} \frac {z_{ia}}{qz_{ib}}
$$
\begin{multline*}
B = \bigoplus_{e = \oij} \Big( \Hom(\BC^{n_i}, \BC^{n_j'}) \oplus \Hom(\BC^{n_j}, \BC^{n_i'}) \Big) \quad \Rightarrow \\ \Rightarrow \quad \chi_{T \times G_{\bn} \times G_{\bn'}}(B) = \sum_{e = \oij \in E} \left( \mathop{\sum_{1\leq a \leq n_i}}_{n_j < b \leq n_j+n_j'} \frac {z_{jb}}{t_ez_{ia}} + \mathop{\sum_{1\leq a \leq n_j}}_{n_i < b \leq n_i+n_i'} \frac {t_ez_{ib}}{qz_{ja}} \right)
\end{multline*}
while:
$$
\chi_{T \times G_{\bn} \times G_{\bn'}}(\CL_{\bn + \bn'}) = \mathop{\prod^{i\in I}_{1\leq a \leq n_i}}_{n_i < b \leq n_i+n_i'} \left( - \frac {z_{ia}}{qz_{ib}} \right)
$$
If a torus representation $V$ has character $\sum_i \chi_i$, then $[\wedge^\bullet V^\vee]$ is equal to $\prod_i (1-\chi_i^{-1})$ in $K$-theory. Therefore, putting the contributions of $A$, $B$, $\CL_{\bn+\bn'}$ together, we get:
\begin{equation}
\label{eqn:temp}
\iota_*(\alpha * \alpha') = \Big( \iota_*(\alpha) \boxtimes \iota_*(\alpha') \Big) \cdot
\end{equation}
$$
\mathop{\prod^{i\in I}_{1\leq a \leq n_i}}_{n_i < b \leq n_i+n_i'} \left(1 - \frac {z_{ia}}{qz_{ib}} \right) \mathop{\prod^{e = \oij \in E}_{1\leq a \leq n_i}}_{n_j < b \leq n_j+n_j'} \left(1 - \frac {t_ez_{ia}}{z_{jb}} \right) \mathop{\prod^{e = \oji \in E}_{1\leq a \leq n_i}}_{n_j < b \leq n_j+n_j'} \left(1 - \frac {qz_{ia}}{t_ez_{jb}} \right)
$$
The expression in the right-hand side matches the right-hand side of \eqref{eqn:shuf prod}, but for two key differences. Firstly, the right-hand side of \eqref{eqn:temp} is missing the factor:
$$
\delta = \mathop{\prod^{i\in I}_{1\leq a \leq n_i}}_{n_i < b \leq n_i+n_i'} \left(1 - \frac {z_{ia}}{z_{ib}} \right)
$$
in the denominator. Secondly, we are missing the symmetrization. The reason for these discrepancies is that \eqref{eqn:temp} is an equality of classes in $K_{T \times G_{\bn} \times G_{\bn'}}(\text{point})$, while $\iota(\alpha * \alpha')$ is defined as a class in $K_{T \times G_{\bn+\bn'}}(\text{point})$. Therefore, one needs to pull the right-hand side of \eqref{eqn:temp} back from $G_{\bn} \times G_{\bn'}$ to a parabolic subgroup $P_{\bn,\bn'} \subset G_{\bn+\bn'}$, and then one has to push-forward the result from $P_{\bn,\bn'}$ to $G_{\bn+\bn'}$. The first operation does not change any formulas, but the second operation is responsible for dividing by the Weyl denominator $\delta$ and the symmetrization. 

\end{proof}

\subsection{} We will now define a certain subalgebra of $\CV_{\text{int}}$, determined by the so-called wheel conditions. These first arose in the context of elliptic quantum groups in \cite{FO}, and the version herein is inspired by the particular wheel conditions of \cite{FHHSY} (which actually correspond to the case when $Q$ is the Jordan quiver in our construction). \\

\begin{definition}
\label{def:shuffle}

The \textbf{shuffle algebra} is defined as the subset:
$$
\CS_{\emph{int}} \subset \CV_{\emph{int}}
$$
of Laurent polynomials $R(\dots, z_{i1}, \dots, z_{in_i}, \dots)$ that satisfy the ``wheel conditions":
\begin{equation}
\label{eqn:wheel}
R \Big|_{z_{ia} = \frac {qz_{jb}}{t_e} = q z_{ic}} = R \Big|_{z_{ja} = t_e z_{ib} = q z_{jc} } = 0
\end{equation}
for all edges $e = \oij$ and all $a \neq c$ (and further $a \neq b \neq c$ if $i = j$). \\

\end{definition}

\noindent The following is elementary (and moreover is closely related to a particular case of Proposition \ref{prop:restricted subalg}) and so we leave its proof as an exercise to the reader. \\

\begin{proposition}
\label{prop:subalg}

$\CS_{\emph{int}}$ is a subalgebra of $\CV_{\emph{int}}$. \\

\end{proposition}

\noindent The following is a key observation, due to Yu Zhao (\cite{Zhao}; although \loccit deals with the case of the Jordan quiver, the generalization to any $Q$ is immediate). \\

\begin{proposition} 
\label{prop:yu}

The image of the map \eqref{eqn:iota} lands in the shuffle algebra:
$$
\iota : K \rightarrow \CS_{\emph{int}} \subset \CV_{\emph{int}}
$$

\end{proposition}

\medskip

\begin{proof} \emph{(closely following Theorem 2.9 and Corollary 2.10 of \cite{Zhao})}: we need to prove that for any $\alpha \in K_{\bn}$, the Laurent polynomial $\iota(\alpha)$ satisfies the wheel conditions \eqref{eqn:wheel}. To this end, let us recall that $\iota$ arises from the closed embedding:
$$
\mu_{\bn}^{-1}(0) \stackrel{i}\hookrightarrow T^* Z_{\bn} = \left\{ \Big( \phi_e : \BC^{n_i} \leftrightharpoons \BC^{n_j} : \phi^*_e \Big)_{\forall e = \oij \in E} \right\}
$$
For any edge $e = \oij$, let us also consider the following locally closed subset:
\begin{equation}
\label{eqn:locally closed}
V_e = \left\{\Big(\phi_e \in \BC^* \cdot E_{bc}, \ \phi_e^* \in \BC^* \cdot E_{ab}, \ \phi_{e'} = \phi^*_{e'} = 0 \ \forall e' \neq e \Big) \right\} \stackrel{j}\hookrightarrow T^*Z_{\bn}
\end{equation}
where $E_{ab}$ denotes the matrix with a single 1 at the intersection of row $a$ and column $b$ (with respect to the standard basis of $\BC^{n_i}$ and $\BC^{n_j}$) and zeroes everywhere else. Because $a \neq c$, it is easy to observe that $\mu_{\bn}$ does not annihilate any point of $V_e$, hence:
$$
\mu_{\bn}^{-1}(0) \cap V_e = \emptyset
$$
and so:
\begin{equation}
\label{eqn:circ 1}
j^*\circ i_* = 0 \qquad \Rightarrow \qquad j^*(i_*(\alpha)) = 0, \quad \forall \alpha \in K_{\bn}
\end{equation}
Let $\pi : T^*Z_{\bn} \rightarrow (\text{point})$ be the usual projection, and thus we have:
\begin{equation}
\label{eqn:circ 2}
i_*(\alpha) = \pi^*(R(\dots,z_{k1},\dots,z_{kn_k},\dots))
\end{equation}
where $R = \iota(\alpha)$. We need to show that $R$ satisfies the wheel conditions with respect to the edge $e$. Formulas \eqref{eqn:circ 1} and \eqref{eqn:circ 2} imply:
$$
j^* ( \pi^*(R)) = 0 \qquad \Rightarrow \qquad \rho^*(R) = 0
$$
where $\rho = \pi \circ j : V_e \rightarrow (\text{point})$ is the usual projection. However, $V_e \cong \BC^* \times \BC^*$, and the action of $T \times (\text{maximal torus of }G_{\bn})$ on the  two copies of $\BC^*$ is given via the characters:
$$
\frac {z_{jb}}{t_e z_{ic}} \qquad \text{and} \qquad \frac {t_e z_{ia}}{q z_{jb}}
$$
respectively. Therefore: 
$$
K_{T \times G_{\bn}}(V_e) = \text{Rep}_T[z^{\pm 1}_{k1},\dots,z_{k n_k}^{\pm 1}]_{k \in I} \Big/ \left(z_{ia} - \frac {qz_{jb}}{t_e}, z_{jb} - t_e z_{ic} \right)
$$
The fact that $\rho^*(R) = 0$ simply means that $R$ lies in the ideal generated by $z_{ia} - \frac {qz_{jb}}{t_e}$ and $z_{jb} - t_e z_{ic}$, which is precisely the first wheel condition in \eqref{eqn:wheel}. To obtain the second wheel condition, one runs the same argument as above, but switching the roles of $E_{ab}$ and $E_{bc}$ in \eqref{eqn:locally closed}. 

\end{proof}

\subsection{} We will now consider the localized versions of the algebras in the previous Subsections, namely:
$$
K_\loc = K \bigotimes_{\text{Rep}_T} \BF \qquad \text{and} \qquad \CV = \CV_{\text{int}} \bigotimes_{\text{Rep}_T} \BF
$$
where $\BF = \BQ(q,t_e)_{e \in E}$ is the fraction field of $\text{Rep}_T$. Define similarly $\CS \subset \CV$ as the subalgebra of Laurent polynomials with $\BF$-coefficients which satisfy the wheel conditions \eqref{eqn:wheel}. We have the following analogue of the algebra homomorphism of Proposition \ref{prop:yu}:
$$
\iota : K_\loc \rightarrow \CS \subset \CV
$$
By construction, $\iota$ preserves the grading by $\bn$. The smallest non-trivial degrees are:
$$
\boldsymbol{\varsigma}_i = (\underbrace{0,\dots,0,1,0,\dots,0}_{1 \text{ on }i\text{-th spot}})
$$
Because $\mu_{\boldsymbol{\varsigma}_i}^{-1}(0)$ is an affine space whose dimension is twice the number of loops at the vertex $i \in I$, it is easy to see that:
\begin{equation}
\label{eqn:degree 1}
K_{\boldsymbol{\varsigma}_i,\loc} \cong K_{T \times \BC^*}(\text{point})_{\text{loc}} \cong \BF[z^{\pm 1}] \cong \CS_{\boldsymbol{\varsigma}_i}
\end{equation}
Let us consider, following \cite{SV}, the spherical subalgebras:
$$
\mathring{K}_\loc \subseteq K_\loc \qquad \text{and} \qquad \oCS \subseteq \CS
$$
which are by definition the subalgebras generated by the graded summands \eqref{eqn:degree 1} for all $i \in I$. Our main result, whose proof will occupy the next Section, is: \\

\begin{theorem}
\label{thm:main}

We have $\oCS = \CS$. \\

\end{theorem}

\begin{remark}

Our main motivation for Theorem \ref{thm:main} is Proposition 1.14 of \cite{SV}, which pertains to the situation when $Q$ is the quiver with one vertex and $g$ loops. In this case, if we let $\oCS_n = \CS_n \cap \oCS$ for any $n \in \BN$, then \loccit showed that:
$$
\oCS_n \supset \CS_n \cdot \prod_{1 \leq i < j \leq n} (z_i-z_j)^{\frac {n!}2} 
$$
In other words, the spherical subalgebra $\oCS$ is sandwiched between the shuffle algebra $\CS$ (which is defined via wheel conditions) and a certain principal ideal within. Theorem \ref{thm:main} shows that one half of this sandwich is an equality. \\

\end{remark}

\subsection{} 

As a consequence of Theorem \ref{thm:main}, we conclude that the map:
\begin{equation}
\label{eqn:iota 2}
\iota : K_{\loc} \rightarrow \CS
\end{equation}
is surjective (simply because the right-hand side is generated by its degree $\boldsymbol{\varsigma}_i$ pieces, as $i$ runs over $I$, and all of these pieces also lie in $K_{\loc}$). However, Varagnolo-Vasserot showed in \cite[Proposition 2.4.4]{VV} that the map $\iota$ is injective \footnote{While \loccit prove the injectivity of $\iota$ for the version of the K-HA supported on a certain nilpotent Lagrangian substack of $T^*\fZ_{\bn}$, the result also holds for the usual K-HA after localization by appealing to Lemma 2.4.1 (d) of \loccit}. We conclude: \\

\begin{corollary}
\label{cor:iso}

The map $\iota$ induces an isomorphism $K_{\eloc} \cong \CS$. \\

\end{corollary}

\noindent At this point, a natural question is whether the non-localized map $K \rightarrow \CS_{\text{int}}$ is also an isomorphism. The answer to this question is, probably, no. Indeed, while Yu Zhao's Proposition \ref{prop:yu} shows that the 3-variable wheel conditions are necessary for describing $\iota(K)$ as a subset of $\CV_{\text{int}}$, it is quite possible that they are not sufficient. In fact, it is possible that wheel conditions involving more than 3 variables exist, and they might be non-redundant in $\CS_{\text{int}}$ and redundant in $\CS$. \\

%for an occurrence of this phenomenon in the closely related situation of shuffle algebras corresponding to quantum loop algebras, see \cite{Ts}, where Tsymbaliuk uses rather involved wheel and divisibility conditions to define an integral form of a quantum group whose localized form can be described just by simple 3-variable wheel conditions). \\

%\begin{remark}
%\label{rem:torus}

%A related issue to that of the previous paragraph is what happens if we restrict the equivariance to a subtorus of $T$, i.e. introduce multiplicative relations between the equivariant parameters $\{q,t_e\}_{e \in E}$. As shown in \cite{VV}, one would need to take care in defining the K-HA and proving the injectivity of the map \eqref{eqn:iota 2}. As for the contents of the subsequent Section (which in the case of maximal equivariance was instrumental in proving the surjectivity of $\iota$) the only part which might fail is the proof of Proposition \ref{prop:pairing}, as specializing the values of $t_e$ might cause the integrals therein to have poles of order $\geq 2$. If this happened, one would need to counter this fact by imposing additional wheel conditions to \eqref{eqn:wheel}, possibly involving more than 3 variables, and then one would need to ensure that Proposition \ref{prop:yu} still holds. \\

%\end{remark}

\section{The shuffle algebra and combinatorics of words}
\label{sec:proof}

\medskip

\subsection{} We will now develop the combinatorial tools that we need to prove Theorem \ref{thm:main}. Many of the ideas herein have been explored in the context of quantum loop groups in \cite{NT}, building upon the work of \cite{LR, L, R}, but we will also introduce some new features that are key to dealing with infinite-dimensional vector spaces. This will allow us to give a proof of Theorem \ref{thm:main} which does not rely on any external features, and has the potential of being applicable to other types of shuffle algebras. \\

\noindent Given our quiver $Q$, let us consider the following symbols for all $i \in I$ and $d \in \BZ$:
\begin{align*}
&e_{i,d} = z_{i1}^d \in \CS_{\bvs_i} \subset \CS
&f_{i,d} = z_{i1}^d \in \CS^{\op}_{\bvs_i} \subset \CS^{\op}
\end{align*}
Recall that $\bvs_i \in \nn$ is the $n$-tuple of integers with a single 1 on the $i$-th position, and zeroes everywhere else. By definition, the $e$'s (respectively the $f$'s) generate the subalgebra $\oCS$ (respectively $\oCS^\op$). For any $\bn = (n_i)_{i \in I} \in \nn$, the number:
$$
n = \sum_{i \in I} n_i
$$
will be called the \textbf{length} of $\bn$. The algebra $\CS$ is graded by $\nn \times \BZ$, with:
$$
\deg R = (\bn,d)
$$
for any $R(\dots,z_{i1},\dots,z_{in_i},\dots) \in \CS$ of homogeneous degree $d$. Let:
$$
\CS_{\bn} \subset \CS
$$
denote the subspace of elements whose degree is contained in $\{\bn\} \times \BZ$. 

\subsection{} 

The following pairing will be one of our main tools. Let $Dz = \frac {dz}{2\pi i z}$. Whenever we write $\int_{|z_1| \ll \dots \ll |z_n|}$ we are referring to a contour integral taken over concentric circles around the origin in the complex plane (i.e. an iterated residue at $0$). \\

\begin{proposition}
\label{prop:pairing}

There is a pairing:
\begin{equation}
\label{eqn:pairing}
\CS \otimes \oCS^{\emph{op}} \xrightarrow{\langle \cdot , \cdot \rangle} \BF
\end{equation}
given for all $R \in \CS_{\bn}$ and all $i_1,\dots,i_n \in I$, $d_1,\dots,d_n \in \BZ$ by: 
\begin{multline}
\label{eqn:pairing formula}
\Big \langle R, f_{i_1,-d_1} * \dots * f_{i_n,-d_n} \Big \rangle = \\ = \int_{|z_1| \ll \dots \ll |z_n|} \frac {z_1^{-d_1}\dots z_n^{-d_n} R(z_1,\dots,z_n)}{\prod_{1\leq a < b \leq n} \zeta_{i_a i_b} \left( \frac {z_a}{z_b} \right)} \prod_{a = 1}^n Dz_a
\end{multline}
if $\bs_{i_1}+\dots+\bs_{i_n} = \bn$, and 0 otherwise (implicit in the notation \eqref{eqn:pairing formula} is that the symbol $z_a$ is plugged into one of the variables $z_{i_a \bullet}$ of $R$, for all $a \in \{1,\dots,n\}$). \\

\noindent Moreover, the pairing \eqref{eqn:pairing} is non-degenerate in the first argument, i.e.:
\begin{equation}
\label{eqn:non-deg}
\Big \langle R, \oCS^{\emph{op}} \Big \rangle = 0 \quad \Rightarrow \quad R = 0
\end{equation}

\medskip

\end{proposition}

\begin{proof} In order to prove that \eqref{eqn:pairing formula} yields a well-defined pairing, we must show that if there exists a linear relation:
\begin{multline}
\label{eqn:have}
\mathop{\sum_{i_1,\dots,i_n \in I}}_{d_1,\dots,d_n \in \BZ} \text{coeff} \cdot f_{i_1,-d_1} * \dots * f_{i_n,-d_n} = 0 \quad \Leftrightarrow \\ \Leftrightarrow \quad \Sym \left[ \mathop{\sum_{i_1,\dots,i_n \in I}}_{d_1,\dots,d_n \in \BZ} \frac { \text{coeff} \cdot  z_1^{-d_1} \dots z_n^{-d_n}}{\prod_{1 \leq a < b \leq n} \zeta_{i_ai_b} \left(\frac {z_a}{z_b} \right)} \right] = 0
\end{multline}
in $\oCS^\op$ (recall that ``Sym" symmetrizes variables $z_a$ and $z_b$ if and only if $i_a = i_b$), then this relation is also reflected in the right-hand side of \eqref{eqn:pairing formula}:
\begin{equation}
\label{eqn:need}
\int_{|z_1| \ll \dots \ll |z_n|} \mathop{\sum_{i_1,\dots,i_n \in I}}_{d_1,\dots,d_n \in \BZ} \frac { \text{coeff} \cdot z_1^{-d_1}\dots z_n^{-d_n}}{\prod_{1\leq a < b \leq n} \zeta_{i_ai_b} \left( \frac {z_a}{z_b} \right)} \cdot R(z_1,\dots,z_n) \prod_{a = 1}^n Dz_a = 0
\end{equation}
for any element $R \in \CS_{\bs_{i_1}+\dots+\bs_{i_n}}$ (implicit in the notation \eqref{eqn:need} is that the symbol $z_a$ is plugged into one of the variables $z_{i_a \bullet}$ of $R$, for all $a \in \{1,\dots,n\}$). The fact that \eqref{eqn:have} implies \eqref{eqn:need} is a particular case of the following statement: \\

\begin{claim}
\label{claim:main}

Consider any rational function of the form:
\begin{equation}
\label{eqn:form of p}
P(z_1,\dots,z_n) = \frac {p(z_1,\dots,z_n)}{\prod_{1\leq a < b \leq n} \zeta_{i_a i_b} \left( \frac {z_a}{z_b} \right)}
\end{equation}
where $p$ is a Laurent polynomial (maybe non-symmetric) which satisfies the wheel conditions in the following form:
\begin{align}
&p \Big|_{z_a = \frac {qz_b}{t_e} = q z_c} = 0 \label{eqn:wheel 1} \\
&p \Big|_{z_a = t_{e'} z_b = q z_c} = 0 \label{eqn:wheel 2}
\end{align}
whenever $a\neq c$ and $i_a = i_c$ (and further $a \neq b \neq c$ if $i_a = i_b = i_c$), and for every edge $e = \overrightarrow{i_ai_b}$ and $e' = \overrightarrow{i_bi_a}$, respectively. Then:
\begin{equation}
\label{eqn:linear functional}
\int_{|z_1| \ll \dots \ll |z_n|} P(z_1,\dots,z_n) \prod_{a=1}^n Dz_a \quad \text{is a linear functional of} \quad \emph{Sym}[P]
\end{equation}

\medskip

\end{claim}

\noindent It remains to prove Claim \ref{claim:main}. In order to do so, let us assume that $\{q,t_e\}_{e \in E}$ are complex numbers whose absolute values satisfy the inequality:
\begin{equation}
\label{eqn:inequality parameters}
|q| < |t_e| < 1 \quad \forall e \in E
\end{equation}
This restriction is not essential for what follows, as all quantities that will arise will be rational functions in $\{q,t_e\}_{e \in E}$. However, it is a useful linguistic device for encoding all the residues that we will encounter in the following argument. Then let us consider the following expressions for all $m \in \{1,\dots,n\}$:
\begin{multline}
X_m = \sum^{\text{fair partition}}_{\{m,\dots,n\} = A_1 \sqcup \dots \sqcup A_t} \int_{|z_1| \ll \dots \ll |z_{m-1}| \ll |z_{\alpha_1}| = \dots = |z_{\alpha_t}|} \\ \Big( \dots \underset{z_{\gamma_s} = z_{\alpha_s} q^2}{\text{Res}} \Big( \underset{z_{\beta_s} = z_{\alpha_s} q}{\text{Res}}  P(z_1,\dots,z_n) \Big) \dots \Big) \prod_{s=1}^t Dz_{\alpha_s} \prod_{a=1}^{m-1} Dz_a \label{eqn:xm}
\end{multline}
In the notation above, we assume that each set $A_s$ that makes up the fair partition is of the form $\{\dots < \gamma_s < \beta_s < \alpha_s\}$. The partition being ``fair" means that each of the sets $A_s$ has the property that all of their elements correspond to one and the same vertex of the quiver, i.e. $i_{\alpha_s} = i_{\beta_s} = i_{\gamma_s} = \dots$ for all $s \in \{1,\dots,t\}$. \\

%\underset{z_{\beta_s} = z_{\alpha_s} q, \ z_{\gamma_s} = z_{\alpha_s} q^2, \ \dots}{\text{Res}}

\begin{claim}
\label{claim:residues}

We have $X_m = X_{m-1}$ for all $m \in \{2,\dots,n\}$. \\

\end{claim} 

\noindent Let us first show how Claim \ref{claim:residues} implies Claim \ref{claim:main}. By iterating Claim \ref{claim:residues} a number of $n-1$ times, we conclude that $X_n = X_1$, or more explicitly:
\begin{equation}
\label{eqn:before}
\int_{|z_1| \ll \dots \ll |z_n|} P(z_1,\dots,z_n) \prod_{a=1}^n Dz_a = 
\end{equation}
$$
=  \sum^{\text{fair partition}}_{\{1,\dots,n\} = A_1 \sqcup \dots \sqcup A_t}  \int_{|z_{\alpha_1}| = \dots = |z_{\alpha_t}|} \Big( \dots \underset{z_{\gamma_s} = z_{\alpha_s} q^2}{\text{Res}} \Big( \underset{z_{\beta_s} = z_{\alpha_s} q}{\text{Res}}  P(z_1,\dots,z_n) \Big) \dots \Big) \prod_{s=1}^t Dz_{\alpha_s}
$$
However, for any fixed fair partition $\{1,\dots,n\} = \bar{A}_1\sqcup \dots \sqcup \bar{A}_t$, we claim that (let us denote $\bar{A}_s = \{\dots < \bar{\gamma}_s < \bar{\beta}_s < \bar{\alpha}_s\}$ for all $s \in \{1,\dots,t\}$ in the formula below):
\begin{equation}
\label{eqn:after}
\int_{|z_{\bar{\alpha}_1}| = \dots = |z_{\bar{\alpha}_t}|} \Big( \dots \underset{z_{\bar{\gamma}_s} = z_{\bar{\alpha}_s} q^2}{\text{Res}} \Big( \underset{z_{\bar{\beta}_s} = z_{\bar{\alpha}_s} q}{\text{Res}}  P(z_1,\dots,z_n) \Big) \dots \Big) \prod_{s=1}^t Dz_{\bar{\alpha}_s} = 
\end{equation}
$$
\mathop{\sum^{\text{fair partition}}_{\{1,\dots,n\} = A_1 \sqcup \dots \sqcup A_t}}_{|A_s| = |\bar{A}_s|, \ i_{\alpha_s} = i_{\bar{\alpha}_s} \ \forall s} \int_{|z_{\alpha_1}| = \dots = |z_{\alpha_t}|} \Big( \dots \underset{z_{\gamma_s} = z_{\alpha_s} q^2}{\text{Res}} \Big( \underset{z_{\beta_s} = z_{\alpha_s} q}{\text{Res}}  P(z_1,\dots,z_n) \Big) \dots \Big) \prod_{s=1}^t Dz_{\alpha_s} 
$$
Indeed, as $\Sym[P]$ sums over all ways to permute the variables of $P$, the left-hand side of \eqref{eqn:after} sums over all iterated residues $z_{\alpha_s} = z_{\beta_s}q^{-1} = z_{\gamma_s}q^{-2} = \dots $ of $P$. The fact that such residues are only non-zero when $\alpha_s > \beta_s > \gamma_s > \dots$ and thus correspond to a fair partition (which yields the right-hand side of \eqref{eqn:after}) is a consequence of the specific denominators of $P$ that appear in \eqref{eqn:form of p}. Having proved \eqref{eqn:after}, the required \eqref{eqn:linear functional} follows from the fact that the right-hand side of equation \eqref{eqn:before} is a linear combination of the right-hand sides of equation \eqref{eqn:after}, and therefore the same holds true for the respective left-hand sides of these equations. \\

\noindent Let us now prove Claim \ref{claim:residues}. To this end, consider the contour integral formula\footnote{Note that our definition of the residue is $-$ the usual one.}:
$$
\int_{|z| \ll |w|} f(z,w) Dz Dw = \int_{|z| = |w|} f(z,w) Dz Dw + \sum_{|c| < 1} \int \left[ \underset{z = wc}{\text{Res}} f(z,w) \right] Dw
$$
for any homogeneous rational function $f$, all of whose poles are simple and of the form $z - wc$. Consider formula \eqref{eqn:xm}, and let us zoom in on the summand corresponding to a given partition $\{m,\dots,n\} = A_1 \sqcup \dots \sqcup A_t$. As we move the (smaller) contour of the variable $z_{m-1}$ toward the (larger) contours of the variables $z_{\alpha_1},\dots,z_{\alpha_t}$, one of two things can happen. The first thing is that the smaller contour reaches the larger ones, which leads to the partition:
$$
\{m-1,\dots,n\} = A_1\sqcup \dots \sqcup A_t \sqcup \{m-1\}
$$
in formula \eqref{eqn:xm} for $m$ replaced by $m-1$. The second thing is that the variable $z_{m-1}$ is ``caught" in a pole of the form $z_{m-1} = z_{\alpha_s} c$ for some $s \in \{1,\dots,t\}$ and some $|c|<1$. However, because the rational function $P$ is of the form \eqref{eqn:form of p}, then:
$$
P(z_1,\dots,z_n) = \frac {p(z_1,\dots,z_n)}{\# \prod_{s=1}^t \zeta_{i_{m-1}i_{\alpha_s}} \left(\frac {z_{m-1}}{z_{\alpha_s}} \right) \zeta_{i_{m-1}i_{\beta_s}} \left(\frac {z_{m-1}}{z_{\beta_s}} \right)\zeta_{i_{m-1}i_{\gamma_s}} \left(\frac {z_{m-1}}{z_{\gamma_s}} \right) ...}
$$
where $\#$ denotes various products of $\zeta$'s which will not be involved in the subsequent argument. As we take the iterated residue in the formula above, we obtain:
$$
\Big( \dots \underset{z_{\gamma_s} = z_{\alpha_s} q^2}{\text{Res}} \Big( \underset{z_{\beta_s} = z_{\alpha_s} q}{\text{Res}}  P(z_1,\dots,z_n) \Big) \dots \Big) = 
$$
$$
= \frac {p(z_1,\dots,z_n)|_{z_{\beta_s} = z_{\alpha_s} q, \ z_{\gamma_s} = z_{\alpha_s} q^2, \dots}}{\#|_{z_{\beta_s} = z_{\alpha_s} q, \ z_{\gamma_s} = z_{\alpha_s} q^2, \dots} \cdot \prod_{s=1}^t \zeta_{i_{m-1}i_{\alpha_s}} \left(\frac {z_{m-1}}{z_{\alpha_s}} \right) \zeta_{i_{m-1}i_{\alpha_s}} \left(\frac {z_{m-1}}{z_{\alpha_s} q} \right)\zeta_{i_{m-1}i_{\alpha_s}} \left(\frac {z_{m-1}}{z_{\alpha_s} q^2} \right) ...}
$$
(recall that $i_{\alpha_s} = i_{\beta_s} = i_{\gamma_s} = \dots$, by the very definition of a fair partition). By looking at the formula for $\zeta$ in \eqref{eqn:def zeta}, we see that the only pole caught in this process is $z_{m-1} = z_{\alpha_s} q^{|A_s|}$, and it only occurs if $i_{m-1} = i_{\alpha_s}$. This happens because: \\

\begin{itemize}[leftmargin=*]

\item all the poles of the form $z_{m-1} = z_{\alpha_s} q^x$ for $x \in \{1,\dots, |A_s|-1\}$ (which only occur if $i_{m-1} = i_{\alpha_s}$) are canceled by the denominators of the $\zeta$ factors \\

\item all the poles of the form $z_{m-1} = z_{\alpha_s}q^x t_e$ for $x \in \{0,\dots,|A_s|-2\}$ and $e = \overrightarrow{i_{\alpha_s}i_{m-1}}$ are canceled by the fact that $p$ satisfies the wheel conditions \eqref{eqn:wheel 1} \\

\item all the poles of the form $z_{m-1} = \frac {z_{\alpha_s} q^x}{t_{e'}}$ for $x \in \{1,\dots,|A_s|-1\}$ and $e' = \overrightarrow{i_{m-1}i_{\alpha_s}}$ are canceled by the fact that $p$ satisfies the wheel conditions \eqref{eqn:wheel 2} \\

\end{itemize}

\noindent \footnote{In the second and third bullets, we don't need to consider the poles corresponding to $x = -1$ and $x = 0$, respectively, due to the inequality \eqref{eqn:inequality parameters} and the fact that we are only looking for poles of the form $z_{m-1} = z_{\alpha_s} c$ for $|c| < 1$.} The residue at the pole $z_{m-1} = z_{\alpha_s} q^{|A_s|}$ leads to the partition:
$$
\{m-1,\dots,n\} = A_1 \sqcup \dots \sqcup A_{s-1} \sqcup \Big( A_s \sqcup \{m-1\} \Big) \sqcup A_{s+1} \sqcup \dots \sqcup A_t
$$
in formula \eqref{eqn:xm} for $m$ replaced by $m-1$. We thus conclude the proof of Claim \ref{claim:residues}, and with it, the proof of Claim \ref{claim:main}. This shows that the pairing \eqref{eqn:pairing} is well-defined. \\

\noindent The non-degeneracy of the pairing \eqref{eqn:pairing} is simply a restatement of the fact that if all coefficients of the power series expansion of the rational function: 
$$
\frac {R(z_1,\dots,z_n)}{\prod_{1 \leq a < b \leq n} \zeta_{i_ai_b} \left( \frac {z_a}{z_b} \right)}
$$
(in the limit $|z_1|\ll\dots\ll|z_n|$) vanish, then $R(z_1,\dots,z_n) = 0$. 

\end{proof}

%\begin{remark} As shown in \cite{SV} (see also \cite{Shuf} for the $g=1$ case), one can extend the algebras $\CS$ and $\CS^{\emph{op}}$ so that they admit topological coproducts. Then \eqref{eqn:pairing} is a bialgebra pairing, and allows one to define the Drinfeld double of the shuffle algebra. \\

%\end{remark}

\subsection{} If $\CS$ were finite-dimensional over $\BF$, then the non-degeneracy of the pairing \eqref{eqn:pairing} in the first argument would imply that $\dim \CS \leq \dim \oCS$. This would be enough to establish Theorem \ref{thm:main}. To adapt this approach to the infinite-dimensional $\CS$, we will decompose it into finite-dimensional pieces, and analyze those. To this end, let us recall some notions from \cite{NT}, which are in turn inspired by the constructions of \cite{LR, L, R} in the setting of Lie algebras and quantum groups. \\

\begin{definition}

Fix a total order on the set $I$ of vertices of the quiver $Q$. This induces a total order on the set of \textbf{letters}: 
$$
i^{(d)}
$$ 
(for all $i \in I$ and $d \in \BZ$) by:
\begin{equation}
\label{eqn:lex affine}
i^{(d)} < j^{(e)} \quad \text{if} \quad 
\begin{cases} d>e \\ \text{ or } \\ d = e \text{ and } i<j \end{cases}
\end{equation}
A \textbf{word} is any sequence of letters:
$$
\left[ i_1^{(d_1)} \dots i_n^{(d_n)} \right] \qquad \forall i_1,\dots,i_n \in I, d_1, \dots, d_n \in \BZ
$$
We have the total lexicographic order on words given by:
$$
\left[ i_1^{(d_1)} \dots i_n^{(d_n)} \right] < \left[ j_1^{(e_1)} \dots j_{m}^{(e_{m})} \right]
$$
if $i_1^{(d_1)} = j_1^{(e_1)}$, \dots, $i_k^{(d_k)} = j_k^{(e_k)}$ and either $i_{k+1}^{(d_{k+1})} < j_{k+1}^{(e_{k+1})}$ or $k = n < m$. \\

\end{definition}

\noindent The \textbf{degree} of a word $v = \left[ i_1^{(d_1)} \dots i_n^{(d_n)} \right]$ is defined as:
$$
\deg v = (\bs_{i_1} + \dots + \bs_{i_n}, d_1+\dots+d_n) \in \nn \times \BZ
$$
its \textbf{sequence of exponents} is defined as
$$
\overline{v} = (d_1,\dots,d_n)
$$
and the \textbf{length} of the word $v$ as above will be the number $n$. \\

\subsection{}

For any word $w = \left[ i_1^{(d_1)} \dots i_n^{(d_n)} \right]$, we will write:
\begin{align}
&e_w = e_{i_1,d_1} * \dots * e_{i_n,d_n} \in \oCS \label{eqn:def ew} \\
&f_w = f_{i_1,-d_1} * \dots * f_{i_n,-d_n} \in \oCS^\op \label{eqn:def fw}
\end{align}
The following is an easy exercise, which we leave to the interested reader: \\

\begin{proposition}
\label{prop:iso}

The assignment $e_w \mapsto f_w$ gives an isomorphism:
$$
\oCS \rightarrow \oCS^{\emph{op}} \Big|_{t_e \mapsto \frac q{t_e} \ \forall e\in E}
$$
(it is easy to see how to extend the map above to the whole of $\CS$, cf. Theorem \ref{thm:main}). \\

%where ``$\CS_Q$" refers to the shuffle algebra with product \eqref{eqn:shuf prod} defined with respect to the quiver $Q$, while $Q^{\emph{op}}$ denotes the quiver where all the arrows are reversed. \\

\end{proposition}

\noindent By definition, elements of $\oCS$ and $\oCS^\op$ are linear combinations of $e_w$'s and $f_w$'s, respectively. One would like to extract a subset of the $e_w$'s and $f_w$'s which determines a basis. To this end, we introduce the following notion. \\

\begin{definition}
\label{def:proto}

A word $v = \left[ i_1^{(d_1)} \dots i_n^{(d_n)} \right]$ is called \textbf{non-increasing} if:
\begin{equation}
\label{eqn:proto}
 i_1^{(d_1)} \geq \dots \geq i_n^{(d_n)} \quad \Leftrightarrow \quad \Big( d_k < d_{k+1} \Big) \text{ or } \Big( d_k = d_{k+1} \text{ and } i_k \geq i_{k+1} \Big), \ \forall k
\end{equation}

\end{definition}

\medskip

\noindent Our first motivation for the Definition above is the following: \\

\begin{proposition}
\label{prop:proto}

For any word $v$, there exist coefficients $\in \BF$ such that:
\begin{align}
&e_v = \sum_{\text{non-increasing words } w \geq v} \emph{coeff} \cdot e_w \label{eqn:linear combination e} \\
&f_v = \sum_{\text{non-increasing words } w \geq v} \emph{coeff} \cdot f_w \label{eqn:linear combination f}
\end{align}
Thus, $\{e_w\}_{w\text{ non-increasing}}$ and $\{f_w\}_{w\text{ non-increasing}}$ span $\oCS$ and $\oCS^{\emph{op}}$, respectively. \\

\end{proposition}

\begin{proof}  We will prove the required statement for the $e$'s, as the statement for the $f$'s follows from Proposition \ref{prop:iso}. Let us consider the formal series: 
$$
e_i(z) = \sum_{d \in \BZ} \frac {e_{i,d}}{z^d}
$$
For any $i,j \in I$, the following relation holds in $\CS_{\bs_i+\bs_j}$-valued formal series in $z$ and $w$, as can easily be seen from the shuffle product formula \eqref{eqn:shuf prod}:
\begin{equation}
\label{eqn:quadratic}
e_i(z) * e_j(w) \zeta_{ji} \left(\frac wz \right) = e_j(w) * e_i(z)  \zeta_{ij} \left(\frac zw \right)
\end{equation}
The meaning of the formula above is that one clears all the denominators of the $\zeta$ functions (which arise if and only if $i = j$) and then identifies the coefficients of the left and right-hand sides in $z$ and $w$. Explicitly, if $i \neq j$ then \eqref{eqn:quadratic} reads:
\begin{multline*}
e_i(z) * e_j(w) \prod_{e = \oji \in E} \left(1 - \frac {t_ew}z \right) \prod_{e = \oij \in E} \left(1 - \frac {qw}{t_e z} \right) = \\ = e_j(w)  * e_i(z)\prod_{e = \oij \in E} \left(1 - \frac {t_ez}w \right) \prod_{e = \oji \in E} \left(1 - \frac {qz}{t_e w} \right)
\end{multline*}
By equating the coefficients of $z^{-a}w^{-b}$ (for any $a,b\in \BZ$) in the formula above, we obtain the following relations in $\CS_{\bs_i+\bs_j}$ (``coeff" denotes various elements of $\BF$):
\begin{equation}
\label{eqn:quadratic 2}
e_{i,a} * e_{j,b} + \sum_{\bullet = 1}^m \text{coeff} \cdot e_{i,a-\bullet} * e_{j,b+\bullet} = e_{j,b} * e_{i,a} + \sum_{\bullet = 1}^m \text{coeff} \cdot e_{j,b-\bullet} * e_{i,a+\bullet}
\end{equation}
where $m = 2|E|$. When $i = j$, formula \eqref{eqn:quadratic} reads:
\begin{multline*}
e_i(z) * e_i(w) (z - wq^{-1}) \prod_{e = \oii \in E} \left[ \left(1 - \frac {t_ew}z \right) \left(1 - \frac {qw}{t_e z} \right) \right] \\ = e_i(w) * e_i(z) (z q^{-1} - w) \prod_{e = \oii \in E} \left[ \left(1 - \frac {t_ez}w \right)\left(1 - \frac {qz}{t_e w} \right) \right]
\end{multline*}
As above, by equating the coefficients of $z^{1-a}w^{-b}$ (for any $a,b\in \BZ$) we obtain:
\begin{multline}
\label{eqn:quadratic 4}
e_{i,a}  * e_{i,b} + \sum_{\bullet=1}^{m+1} \text{coeff} \cdot e_{i,a-\bullet} * e_{i,b+\bullet} = \\ = - e_{i,b+1} * e_{i,a-1}  + \sum_{\bullet=1}^{m+1} \text{coeff} \cdot e_{i,b+1-\bullet} * e_{i,a-1 + \bullet} 
\end{multline}
We will use the formulas above to prove the following stronger version of \eqref{eqn:linear combination e}, by induction over $n$. There exists a number $\beta(n)$, which only depends on $n \in \BN$, such that for all $i_1,\dots,i_n \in I$ and $d_1,\dots,d_n \in \BZ$ we have: 
\begin{equation}
\label{eqn:proto strong}
e_{i_1,d_1} * \dots * e_{i_n,d_n} = 
\end{equation}
$$
= \mathop{\sum_{\text{non-increasing words } \left[ j_1^{(r_1)} \dots j_n^{(r_n)} \right] \geq \left[ i_1^{(d_1)} \dots i_n^{(d_n)} \right]}}_{\text{such that } \min(d_k) - \beta(n) \leq \min(r_k) \leq \max(r_k) \leq \max(d_k) + \beta(n)} \text{coeff } \cdot e_{j_1,r_1} * \dots * e_{j_n,r_n}
$$
The case $n=1$ is trivial, so let us start with the case $n=2$. If $i_1^{(d_1)} \geq i_2^{(d_2)}$ there is nothing to prove, while if $i_1^{(d_1)} < i_2^{(d_2)}$, we can use \eqref{eqn:quadratic 2} or \eqref{eqn:quadratic 4} to replace:
\begin{equation}
\label{eqn:products}
e_{i_1,d_1} * e_{i_2,d_2} \quad \text{by} \quad e_{i_1,d_1-x} * e_{i_2,d_2+x} \quad \text{and} \quad e_{i_2,d_2-y} * e_{i_1,d_1+y}
\end{equation}
for various $1 \leq x \leq m+1$ and $\delta_{i_2}^{i_1}(-1+\delta_{d_2+1}^{d_1}) \leq y \leq m$. The latter products in \eqref{eqn:products} are already non-increasing, and we may simply repeat the process for the middle products in \eqref{eqn:products} if $d_1-x>d_2+x$ or if $d_1-x=d_2+x$ and $i_1<i_2$. It is easy to see that we thus obtain the $n=2$ case of formula \eqref{eqn:proto strong}, with $\beta(2) = m+1$. \\

\noindent Now let us take any $n \geq 3$, and assume that \eqref{eqn:proto strong} holds for $1,2,\dots,n-1$. We will play the following game with the product of $e$'s in the left-hand side of \eqref{eqn:proto strong}: \\

\begin{itemize}

\item perform a pre-move, i.e. apply \eqref{eqn:proto strong} to $e_{i_1,d_1} * \dots * e_{i_{n-1},d_{n-1}}$, \\

\item in every summand of the resulting expression, perform a post-move, i.e. apply \eqref{eqn:proto strong} to $e_{j_2,r_2} * \dots * e_{j_n,r_n}$, \\

\item in every summand of the resulting expression, go back to the first bullet and perform a pre-move, and so on. \\

\end{itemize}

\noindent This game branches out like a tree, since at every step we choose a summand from a finite linear combination of products of $e$'s, and apply either a pre-move or a post-move. In every branch, we stop the game whenever we encounter an expression:
\begin{equation}
\label{eqn:mon f}
e_{k_1,x_1} * \dots * e_{k_n,x_n}
\end{equation}
with $w = [k_1^{(x_1)}\dots k_n^{(x_n)}]$ a non-increasing word, as any further pre-moves or post-moves would be trivial. But if $w$ as above is not non-increasing, then any pre-move will make $k_1^{(x_1)}$ strictly increase, while any post-move will make $k_n^{(x_n)}$ strictly decrease (this is an easy consequence of relations \eqref{eqn:quadratic 2} and \eqref{eqn:quadratic 4}). Since there are finitely many choices for $k_1$ and $k_n$ as elements of $I$, this means that after finitely many moves the exponent $x_1$ must strictly decrease and the exponent $x_n$ must strictly increase. Thus, we can divide every branch of the game into two parts: \\

\noindent \emph{Part I: while $x_1 \geq x_n$}. By the preceding two sentences, this part of the game can last at most $(d_1 - d_n) \cdot |I|$ moves. We want to show that any exponents $x_1,\dots,x_n$ encountered in \eqref{eqn:mon f} during this part are sandwiched between $\min(d_1,\dots,d_n)$ and $\max(d_1,\dots,d_n)$. Indeed, let's see that this property is preserved under a pre-move (the case of a post-move is analogous, and we leave it to the interested reader). Since a pre-move\footnote{Except for the very first pre-move in the game, but that one can only increase the exponents by a bounded amount, so it does not affect our overall argument} comes after a post-move, then just before the pre-move we have:
$$
x_2 \leq \dots \leq x_n
$$
By assumption, the maximum of the $x_k$'s (namely $x_1$) and the minimum of the $x_k$'s (namely $x_2$) are contained between $\min(d_1,\dots,d_n)$ and $\max(d_1,\dots,d_n)$. After the pre-move, the exponents will change according to:
$$
(x_1,x_2,\dots,x_{n-1}) \mapsto (x_1' \leq x_2' \leq \dots \leq x_{n-1}')
$$
such that $x_1 + \dots + x_{n-1} = x_1' + \dots + x_{n-1}'$, because moves preserve the sum of the exponents involved (see \eqref{eqn:quadratic 2}, \eqref{eqn:quadratic 4}). If we are still in Part I after the pre-move, this means that $x_1' \geq x_n$, so all the numbers $x_1',\dots,x_{n-1}'$ are still $\geq \min(d_1,\dots,d_n)$. Before the pre-move, only the number $x_1$ was greater than $x_n$, while after the pre-move, all the numbers $x_1',\dots,x_{n-1}'$ are greater than $x_n$. This is only possible if the numbers $x_1',\dots,x_{n-1}'$ are no greater than $x_1$, so they will still be $\leq \max(d_1,\dots,d_n)$. \\

\noindent \emph{Part II: while $x_1 < x_n$}. In this case, the values of $x_1,\dots,x_{n}$ can become greater then the maximum (respectively lower than the minimum) of $d_1,\dots,d_n$ as we perform the two kinds of moves. However, by the induction hypothesis, at each step the values of $x_1,\dots,x_{n}$ can only become larger (respectively smaller) by $\beta(n-1)$ than the maximum (respectively minimum) of the analogous values at the previous step in the game. Let us perform a bounded number of moves, until we have:
$$
x_n - x_1 > 2n \cdot \max(\beta(1),\beta(2),\dots,\beta(n-1))
$$ 
and so the values of all the exponents $x_1,\dots,x_n$ can only become greater than the maximum (respectively lower than the minimum) of $d_1,\dots,d_n$ by a fixed amount. Let's assume a summand \eqref{eqn:mon f} was obtained after a pre-move, which means that:
$$
k_1^{(x_1)} \geq \dots \geq k_{n-1}^{(x_{n-1})}
$$
and in particular entails the inequalities $x_1 \leq \dots \leq x_{n-1}$. By our assumption on the size of the difference $x_n - x_1$, there exists a number $s \in \{1,\dots,n-1\}$ such that $\min(x_n,x_{s+1}) - x_s > 2\max(\beta(1),\beta(2),\dots,\beta(n-1))$. If $s = n-1$, then the monomial \eqref{eqn:mon f} already corresponds to a non-increasing word, and we are done. If $s < n-1$, then we simply apply the induction hypothesis of \eqref{eqn:proto strong} to:
$$
e_{k_1,x_1}  * \dots * e_{k_s,x_s} \qquad \text{and} \qquad e_{k_{s+1},x_{s+1}} * \dots * e_{k_n,x_n}
$$
and we conclude that the expressions above are equal to linear combinations of non-increasing words:
$$
e_{l_1,y_1} * \dots * e_{l_s,y_s} \qquad \text{and} \qquad e_{l_{s+1},y_{s+1}} * \dots * e_{l_n,y_n}
$$
respectively, where: 
$$
y_s \leq x_s + \beta(s) \qquad \text{and} \qquad y_{s+1} \geq \min(x_n,x_{s+1}) - \beta(n-s)
$$
These inequalities force $y_{s+1} > y_s$, which implies that the word $[l_1^{(y_1)} \dots l_n^{(y_n)}]$ is non-increasing, and we are done. A similar analysis applies to the situation that \eqref{eqn:mon f} was obtained after a post-move, and we leave the details to the interested reader. At the end of the game, the values of the exponents $y_1,\dots,y_n$ can only become greater than the maximum (respectively lower than the minimum) of $d_1,\dots,d_n$ by a bounded amount, so the proof of the induction step is complete.

\end{proof}

\subsection{} Another reason for considering non-increasing words is the following: \\

\begin{lemma}
\label{lem:finite}

There are finitely many non-increasing words of given degree, which are bounded above by any given word $v$. \\

\end{lemma} 

\begin{proof} Let us assume we are counting non-increasing words $[i_1^{(d_1)} \dots i_n^{(d_n)}]$ with $d_1+\dots+d_n = d$ for fixed $n$ and $d$. The fact that such words are bounded above implies that $d_1$ is bounded below. But then the inequality \eqref{eqn:proto} implies that $d_2,\dots,d_n$ are also bounded below. The fact that $d_1+\dots+d_n$ is fixed implies that there can only be finitely many choices for the exponents $d_1,\dots,d_n$. Since there are also finitely many choices for $i_1,\dots,i_n \in I$, this concludes the proof.  

\end{proof}

\noindent The following notion is inspired by the construction of \cite{LR,L,R} in the case of Lie algebras and finite type quantum groups, and \cite{NT} in the case of quantum loop groups. \\

\begin{definition}
\label{def:standard}

A word $v$ is called \textbf{standard} if $e_v$ cannot be written as a linear combination of $e_w$ for various $w > v$. \\
\end{definition}

\noindent As a consequence of \eqref{eqn:linear combination e}, we see that any standard word is non-increasing, and that we would get the same notion of standard words if we inserted the word ``non-increasing" after the word ``various" in Definition \ref{def:standard}. It would be very interesting to develop a combinatorial description of standard words (see \cite{NT} for the case when $Q$ is a Dynkin diagram of finite type). \\

\subsection{} We will now compute how the elements $e_v$ and $f_w$ pair with each other under \eqref{eqn:pairing}, for various words $v$ and $w$ of the same degree. We will write:
\begin{equation}
\label{eqn:form}
v = \left[ i_1^{(d_1)} \dots i_n^{(d_n)} \right] \quad \text{and} \quad w = \left[ j_1^{(k_1)} \dots j_n^{(k_n)} \right] 
\end{equation}
By formula \eqref{eqn:pairing formula}, $\langle e_v, f_w \rangle$ equals:
$$
\int_{|z_1| \ll \dots \ll |z_n|} \frac {z_1^{-k_1}\dots z_n^{-k_n}}{\prod_{1\leq a < b \leq n} \zeta_{j_a j_b} \left( \frac {z_a}{z_b} \right)} \cdot \Sym \left[ x_1^{d_1} \dots x_n^{d_n} \prod_{1 \leq a < b\leq n} \zeta_{i_a i_b} \left(\frac {x_a}{x_b} \right) \right] \prod_{a = 1}^n Dz_a
$$
A little explanation is in order to make sense of the expression above. To write $e_v$ as $\text{Sym}[\dots]$ in this expression, we are implicitly plugging the variable $x_a$ instead of one of the variables $z_{i_a\bullet}$ of the Sym, for all $a \in \{1,\dots,n\}$. However, to apply formula \eqref{eqn:pairing formula}, the variable $z_a$ must be identified with one of the variables $z_{j_a \bullet}$ of the Sym. Therefore, we are compelled to identify $x_a = z_{\sigma(a)}$ for some permutation $\sigma \in S(n)$ which satisfies $i_a = j_{\sigma(a)}$ for all $a \in \{1,\dots,n\}$. Put differently, the symbol Sym in the expression above must be interpreted as summing only over those permutations $\sigma \in S(n)$ such that $i_a = j_{\sigma(a)}$ for all $a \in \{1,\dots,n\}$. We conclude that:
\begin{multline}
\label{eqn:pair e and f}
\Big \langle e_v, f_w \Big \rangle = \int_{|z_1| \ll \dots \ll |z_n|} \frac {z_1^{-k_1}\dots z_n^{-k_n}}{\prod_{1\leq a < b \leq n} \zeta_{j_a j_b} \left( \frac {z_a}{z_b} \right)} \\ \mathop{\sum_{\sigma \in S(n)}}_{i_a = j_{\sigma(a)} \ \forall a} \left[ z_{\sigma(1)}^{d_1} \dots z_{\sigma(n)}^{d_n} \prod_{1 \leq a < b\leq n} \zeta_{i_a i_b} \left(\frac {z_{\sigma(a)}}{z_{\sigma(b)}} \right) \right] \prod_{a = 1}^n Dz_a =
\end{multline}
$$
= \int_{|z_1| \ll \dots \ll |z_n|} \mathop{\sum_{\sigma \in S(n)}}_{i_a = j_{\sigma(a)} \ \forall a} z_1^{d_{\sigma^{-1}(1)}-k_1}\dots z_n^{d_{\sigma^{-1}(n)}-k_n}  \prod^{a<b}_{\sigma^{-1}(a) > \sigma^{-1}(b)} \frac {\zeta_{j_b j_a} \left(\frac {z_{b}}{z_{a}} \right)}{\zeta_{j_a j_b} \left(\frac {z_{a}}{z_{b}} \right)} \prod_{a = 1}^n Dz_a
$$
Let $\#_{\oij}$ denote the number of arrows from $i$ to $j$, and:
\begin{equation}
\label{eqn:def sharp}
\#_{ij} = \#_{\oij} + \#_{\oji}
\end{equation}
Thus, $\#_{ij}$ counts the total number of edges between $i\neq j$, and twice the number of loops at $i$ if $i=j$. Because of the easily seen fact that:
\begin{equation}
\label{eqn:ratio zero}
\frac {\zeta_{ij}(x^{-1})}{\zeta_{ji}(x)} \in x^{-\#_{ij}} \cdot \BF[[x]]
\end{equation}
formula \eqref{eqn:pair e and f} implies that:
\begin{equation}
\label{eqn:non-zero}
\Big \langle e_v, f_w \Big \rangle \neq 0 \qquad \Rightarrow  
\end{equation}
$$
(k_1,\dots,k_n) = (d_{\sigma^{-1}(1)}, \dots, d_{\sigma^{-1}(n)}) + \sum^{a < b}_{\sigma^{-1}(a) > \sigma^{-1}(b)} c_{a,b} \cdot \underbrace{(0,\dots,1,\dots,-1,\dots,0)}_{1\text{ on position }a, -1 \text{ on position }b}
$$
for some $\sigma \in S(n)$ such that $i_a = j_{\sigma(a)}, \ \forall a$ and some $\{c_{a,b} \geq - \#_{j_aj_b}\}^{a < b}_{\sigma^{-1}(a) > \sigma^{-1}(b)}$. \\

\begin{remark}
\label{rem:symmetry}

Let us prove an ``almost" symmetry property for the pairing. If we change the variables to $y_a = z_{\sigma(a)}$ in \eqref{eqn:pair e and f}, we may conclude that $\langle e_v, f_w\rangle$ equals:
$$
\int_{|y_{\sigma^{-1}(1)}| \ll \dots \ll |y_{\sigma^{-1}(n)}|} \mathop{\sum_{\sigma \in S(n)}}_{i_a = j_{\sigma(a)} \ \forall a} y_1^{d_1-k_{\sigma(1)}} \dots y_n^{d_n-k_{\sigma(n)}}  \prod^{a<b}_{\sigma(a) > \sigma(b)} \frac {\zeta_{i_a i_b} \left(\frac {y_{a}}{y_{b}} \right)}{\zeta_{i_b i_a} \left(\frac {y_{b}}{y_{a}} \right)} \prod_{a = 1}^n Dy_a
$$
The contours of integration are such that $|y_a| \ll |y_b|$ if and only if $\sigma(a) < \sigma(b)$. This means that we can move the contours to ensure that $|y_1| \gg \dots \gg |y_n|$ without picking up any new poles, so we conclude:
\begin{equation}
\label{eqn:symmetric pairing}
\Big \langle e_{i_1,d_1} * \dots * e_{i_n,d_n}, R \Big \rangle = \int_{|y_1| \gg \dots \gg |y_n|} \frac {y_1^{d_1}\dots y_n^{d_n} R(y_1,\dots,y_n)}{\prod_{1\leq a < b \leq n} \zeta_{i_b i_a} \left( \frac {y_b}{y_a} \right)} \prod_{a = 1}^n Dy_a
\end{equation}
where $R = f_{j_1,-k_1} * \dots * f_{j_n,-k_n} \in \oCS^{\emph{op}}$ (in the formula above, the symbol $y_a$ is plugged into one of the variables $z_{i_a \bullet}$ of $R$, for all $a$). As soon as we prove Theorem \ref{thm:main}, we will obtain $\oCS^{\emph{op}} = \CS^{\emph{op}}$, so formula \eqref{eqn:symmetric pairing} will actually hold for all elements $R \in \CS^{\emph{op}}$. Comparing \eqref{eqn:pairing formula} with \eqref{eqn:symmetric pairing} reveals the ``almost" symmetry of the pairing: 
\begin{equation}
\label{eqn:pairing full}
\CS \otimes \CS^{\emph{op}} \xrightarrow{\langle \cdot,\cdot \rangle} \BF
\end{equation}

\end{remark}

\subsection{} 
\label{sub:graph}

Let $m = 2|E|$, and consider the infinite graph $G$ whose vertices are all the non-decreasing sequences of integers $(d_1 \leq \dots \leq d_n)$, and edges are:
\begin{equation}
\label{eqn:edge}
(d_1\leq \dots \leq d_n) \ \longrightarrow \ (d_1' \leq \dots \leq d_n') 
\end{equation}
if:
\begin{equation}
\label{eqn:edge 2}
d_a' =  d_{\sigma(a)} - \sum^{s < a}_{\sigma(s) > \sigma(a)} c_{s,a} + \sum^{a < t}_{\sigma(a) > \sigma(t)} c_{a,t}, \qquad \forall a \in \{1,\dots,n\}
\end{equation}
for some permutation $\text{Id} \neq \sigma \in S(n)$ and some collection of non-negative integers $\{c_{a,b} \geq -m\}^{a < b}_{\sigma(a) > \sigma(b)}$. While a priori a directed graph, $G$ can actually be considered to be undirected, because the existence of a left-to-right edge in \eqref{eqn:edge} also implies the existence of the corresponding right-to-left edge, with respect to:
$$
\sigma' = \sigma^{-1} \qquad \text{and} \qquad c'_{a,b} = c_{\sigma^{-1}(b),\sigma^{-1}(a)}
$$
$$$$

\begin{lemma}
\label{lem:combi}

All connected components of $G$ are finite. \\

\end{lemma}

\noindent Lemma \ref{lem:combi} is a combinatorial statement (or a statement in the theory of root systems, see Remark \ref{rem:root}), which we will prove at the very end of the present Section. \\

\subsection{} 
\label{sub:direct sum}

In formula \eqref{eqn:non-zero}, we showed that if $v$ and $w$ are non-increasing words, then the pairing $\langle e_v,f_w\rangle$ vanishes unless $\overline{v}$ and $\overline{w}$ are connected by an edge in $G$, where $\overline{v}$ denotes the sequence of exponents of the word $v$. Because of this, for any connected component $H \subset G$, we may define the \underline{finite-dimensional} subspaces:
\begin{align*}
&\oCS_H = \sum_{w \text{ non-increasing}}^{\overline{w} \in H} \BF \cdot e_w \\
&\oCS^{\op}_H = \sum_{w \text{ non-increasing}}^{\overline{w} \in H} \BF \cdot f_w 
\end{align*}
of $\oCS$ and $\oCS^{\op}$, respectively. As we have just explained, we have:
\begin{equation}
\label{eqn:pair non-zero}
\left \langle \oCS_H, \oCS^{\op}_{H'} \right \rangle = 0
\end{equation}
for any distinct connected components $H \neq H'$ of $G$. Therefore, because the pairing \eqref{eqn:pairing} is non-degenerate in the first argument, then so is its restriction to:
\begin{equation}
\label{eqn:restricted pairing}
\oCS_H \otimes \oCS_H^{\op} \xrightarrow{\langle \cdot , \cdot \rangle} \BF
\end{equation}
for any connected component $H \subset G$. Switching the roles of $\oCS_H$ and $\oCS_H^{\op}$ (see Remark \ref{rem:symmetry}) implies the non-degeneracy of \eqref{eqn:restricted pairing} in the second argument as well. \\

\begin{proposition}
\label{prop:direct}

For any $n\in \BN$, we have:
\begin{equation}
\label{eqn:direct sum 1}
\bigoplus_{\text{length}(\bn)=n} \oCS_{\bn} = \bigoplus_{H \text{ a connected component of } G} \oCS_H
\end{equation}
and:
\begin{equation}
\label{eqn:direct sum 2}
\oCS_H = \bigoplus_{w \text{ standard}}^{\overline{w} \in H} \BF \cdot e_w
\end{equation}
as well as the analogous statements for $\oCS^{\emph{op}}$. \\

\end{proposition}

\begin{proof} Because the $e_w$'s span $\oCS$ as $w$ runs over all non-increasing words, all that we need to do to prove \eqref{eqn:direct sum 1} is to show that there are no linear relations among the various direct summands of the RHS. To this end, assume that we had a relation:
$$
\sum_{H \text{ a connected component of } G} \alpha_H = 0
$$
for various $\alpha_H \in \oCS_H$. Pairing the relation above with a given $\oCS^{\op}_H$ implies that:
$$
\left \langle \alpha_H, \oCS^{\op}_H \right \rangle = 0
$$
Because the pairing \eqref{eqn:restricted pairing} is non-degenerate, this implies that $\alpha_H = 0$. As for \eqref{eqn:direct sum 2}, it holds because any vector space spanned by vectors $\alpha_1,\dots,\alpha_k$ has a basis consisting of those $\alpha_i$'s which cannot be written as linear combinations of $\{\alpha_j\}_{j > i}$. \\

\end{proof}

\subsection{} 
\label{sub:proof}

We are now ready to prove our main Theorem. \\

\begin{proof} \emph{of Theorem \ref{thm:main}:} Consider any $R \in \CS_{\bn}$. From \eqref{eqn:pairing formula}, it is easy to see that:
$$
\left \langle R, f_{[i_1^{(d_1)} \dots i_n^{(d_n)} ]} \right \rangle = 0
$$
if $d_1$ is small enough. However, by Lemma \ref{lem:finite}, there are only finitely many non-increasing words $w$ of given degree with $d_1$ bounded below. This implies that:
$$
\Big \langle R, f_w \Big \rangle \neq 0
$$
only for finitely many non-increasing words $w$. Letting $H_1, \dots, H_t \subset G$ denote the connected components which contain the sequences of exponents of the aforementioned words, then \eqref{eqn:pair non-zero} and the non-degeneracy of the pairings \eqref{eqn:restricted pairing} imply that there exists an element:
$$
R' \in \oCS_{H_1} \oplus \dots \oplus \oCS_{H_t} \subset \oCS
$$
such that $\langle R, f_w \rangle = \langle R', f_w \rangle$ for all non-increasing words $w$. Then the non-degeneracy statement \eqref{eqn:non-deg} implies that $R = R' \in \oCS$, as we needed to prove. 

\end{proof}

\subsection{} 
\label{sub:pbw} 

As a consequence of Theorem \ref{thm:main} and \eqref{eqn:direct sum 1}--\eqref{eqn:direct sum 2}, we have:
\begin{align}
&\CS = \bigoplus_{w \text{ standard}} \BF \cdot e_w \label{eqn:pbw e} \\
&\CS^{\op} = \bigoplus_{w \text{ standard}} \BF \cdot f_w \label{eqn:pbw f}
\end{align}
Even though the vector spaces $\CS$ and $\CS^{\op}$ are infinite-dimensional, the fact that they arise as direct sums of finite-dimensional vector spaces \eqref{eqn:direct sum 1}--\eqref{eqn:direct sum 2} which are mutually orthogonal under the pairing, allows us to define the dual bases:
\begin{align}
&\CS = \bigoplus_{w \text{ standard}} \BF \cdot e^w \label{eqn:pbw ee} \\
&\CS^{\op} = \bigoplus_{w \text{ standard}} \BF \cdot f^w \label{eqn:pbw ff}
\end{align} 
In other words, we have by definition:
\begin{equation}
\label{eqn:dual bases}
\Big \langle e^v, f_w \Big \rangle = \Big \langle e_v, f^w \Big \rangle = \delta^v_w
\end{equation}
for all standard words $v$ and $w$. \\

\begin{definition}

Any non-zero $R \in \CS$ can be written in the form:
\begin{equation}
\label{eqn:leading word monomials}
R = \dots + \text{constant} \cdot z_{i_1a_1}^{d_1} \dots z_{i_na_n}^{d_n} + \dots
\end{equation}
where we order the variables in any monomial above such that the word:
\begin{equation}
\label{eqn:lexicographically largest word}
\left[ i_1^{(d_1)} \dots i_n^{(d_n)} \right]
\end{equation}
is non-increasing. The \textbf{leading word} of any non-zero $R \in \CS$ is the lexicographically largest word \eqref{eqn:lexicographically largest word} among all the constituent monomials of $R$ in \eqref{eqn:leading word monomials}. \\

\end{definition}

\noindent Leading words are always non-increasing in the sense of \eqref{eqn:proto}.  The following is a straightforward consequence of \eqref{eqn:pairing formula}, which we leave as an exercise to the reader. \\

\begin{proposition}
\label{prop:leading pairing}

A non-zero element $R \in \CS$ has leading word $v$ if and only if:
\begin{equation}
\label{eqn:standard eq}
\Big \langle R, f_w \Big \rangle \text{ is } \begin{cases} \neq 0 &\text{if }w = v \\ = 0 &\text{if }w > v \end{cases}
\end{equation}

\end{proposition}

\medskip

\noindent Since any $f_w$ is a linear combination of $f_y$'s for standard $y\geq w$, then \eqref{eqn:dual bases} implies that $\langle e^v, f_w \rangle = 0$ for all words $w > v$. Therefore, Proposition \ref{prop:leading pairing} implies that:
\begin{equation}
\label{eqn:leading word}
e^v \text{ has leading word }v
\end{equation}
for all standard words $v$. \\

\begin{proposition}
\label{prop:standard is leading}

A word $v$ is standard if and only if it is the leading word of some non-zero $R \in \CS$. \\

\end{proposition}

\begin{proof} The ``if"  implication follows from \eqref{eqn:standard eq}, as it precludes $f_v$ from being a linear combination of $f_w$ with $w > v$, while the ``only if" implication was proved by \eqref{eqn:leading word}. Note that we are tacitly identifying the notion of standard words for $\CS$ and $\CS^{\op}$, i.e. saying that $e_v$ is a linear combination of $e_w$ with $w > v$ if and only if $f_v$ is a linear combination of $f_w$ with $w > v$, which is allowed due to Proposition \ref{prop:iso}. 

\end{proof}

\noindent The Proposition above tells us how to recursively express any $R \in \CS$ in the basis \eqref{eqn:pbw ee}: let $\alpha$ be the coefficient of the leading word monomial of $R$ (call the leading word $v$); then the leading word of $R' = R - \alpha e^v$ is strictly smaller than $v$, and we repeat the process. This terminates after finitely many steps due to Lemma \ref{lem:finite}. \\

%\begin{proposition}
%\label{prop:subword}

%Any subword of a standard word is standard. \\

%\end{proposition}

%\begin{proof} If a subword $v' = [i_{k+1}^{(d_{k+1})} \dots i_{l-1}^{(d_{l-1})}]$ of a word $v = [i_1^{(d_1)} \dots i_{n}^{(d_n)}]$ failed to be standard, then we would have a linear relation:
%$$
%f_{v'} = \sum_{w' > v'} \text{coeff} \cdot f_{w'}
%$$
%Multiplying this relation on the left with $f_{[i_1^{(d_1)} \dots i_k^{(d_k)}]}$ and on the right with $f_{[i_l^{(d_l)} \dots i_n^{(d_n)}]}$ would imply a similar relation for $f_v$, thus contradicting the fact that $v$ is standard. 

%\end{proof}

\subsection{} We still owe the reader a proof of Lemma \ref{lem:combi}. \\

\begin{proof} \emph{of Lemma \ref{lem:combi}:} Let us consider an edge between two sequences:
\begin{equation}
\label{eqn:edge 3}
(d_1 \leq \dots \leq d_n) \quad \longrightarrow \quad (d'_1 \leq \dots \leq d'_n)
\end{equation}
in the graph $G$, where we assume that the two sequences are related by \eqref{eqn:edge 2}. Let us see what the existence of such an edge says about the sequence $(d_1,\dots,d_n)$ in relation to the permutation $\sigma$. For all $a < b$ we have:
$$
d_{\sigma(a)} - \sum^{s < a}_{\sigma(s) > \sigma(a)} c_{s,a} + \sum^{a < t}_{\sigma(a) > \sigma(t)} c_{a,t} = d_a' \leq d_b'  = d_{\sigma(b)} - \sum^{s < b}_{\sigma(s) > \sigma(b)} c_{s,b} + \sum^{b < t}_{\sigma(b) > \sigma(t)} c_{b,t}
$$
Let us consider a pair $a < b$ such that $\sigma(a) > \sigma(b)$, which is maximal in the sense that any $s < a$ has the property that $\sigma(s) < \sigma(a)$ and any  $t > b$ has the property that $\sigma(t) > \sigma(b)$. The inequality in the display above then reads:
$$
d_{\sigma(b)} - d_{\sigma(a)} \geq \sum^{s < b}_{\sigma(s) > \sigma(b)} c_{s,b} - \sum^{s < a}_{\sigma(s) > \sigma(a)} c_{s,a} + \sum^{a < t}_{\sigma(a) > \sigma(t)} c_{a,t} - \sum^{b < t}_{\sigma(b) > \sigma(t)} c_{b,t}
$$
By the maximality assumption of the pair $a < b$, the two sums with minus signs in front are vacuous, and from the assumption $c_{a,b} \geq -m$ for all $a,b$ we infer that:
\begin{equation}
\label{eqn:ineq other way}
d_{\sigma(b)} - d_{\sigma(a)} \geq -2mn
\end{equation}
For any given $k \in \{1,\dots,n-1\}$, assume that $\sigma$ does not send the set $\{1,\dots,k\}$ to itself. Then there exist numbers $a$ and $b$ such that $a \leq k < b$ and $\sigma(b) \leq k < \sigma(a)$. Moreover, we may choose the pair $a < b$ maximal, and so formula \eqref{eqn:ineq other way} applies. However, the fact that $d_1 \leq \dots \leq d_n$ implies that:
\begin{equation}
\label{eqn:necessary}
d_k - d_{k+1} = \underbrace{d_{\sigma(b)} - d_{\sigma(a)}}_{\geq - 2mn} + \underbrace{d_k - d_{\sigma(b)}}_{\geq 0} + \underbrace{d_{\sigma(a)} - d_{k+1}}_{\geq 0} \geq - 2mn
\end{equation}
Therefore, the only $k$ for which we might have $d_k - d_{k+1} < -2mn$ are those for which $\sigma$ sends the set $\{1,\dots,k\}$ to itself and the set $\{k+1,\dots,n\}$ to itself. \\

\noindent We are now ready to prove the following statement by induction on $n$: \underline{there exists a} \underline{natural number $\gamma(n)$ such that two sequences $(d_1 \leq \dots \leq d_n)$ and $(d'_1 \leq \dots \leq d'_n)$} \\
\underline{are connected by a path in $G$ only if $|d_1-d'_1| \leq \gamma(n)$ and $|d_n - d_n'| \leq \gamma(n)$}. This statement implies Lemma \ref{lem:combi}, because for any fixed $d_1,\dots,d_n$, there exist finitely many sequences $d_1' \leq \dots \leq d_n'$ which have $d_1'$ bounded below and $d_n'$ bounded above. The base case of the induction is vacuous, as we can take $\gamma(1) = 0$. For the induction step, assume that $\gamma(1),\dots,\gamma(n-1)$ have been constructed, and define: 
$$
\gamma(n) = (n-1) \cdot \max_{1 \leq k \leq n-1} [\gamma(k)+\gamma(n-k)+2mn]
$$ 
Assume for the purpose of contradiction that the two sequences $(d_1,\dots,d_n)$ and $(d_1',\dots,d_n')$ are connected in the graph $G$, all the while $d_1 < d'_1 - \gamma(n)$ (the situation when $d_1 > d'_1 + \gamma(n)$ is proved by switching the roles of $d_k$ and $d_k'$, and the situation when $|d_n - d_n'| >  \gamma(n)$ is analogous, and so left to the interested reader). Because two sequences connected by a path in $G$ have the same average, we have:
$$
\min(d_1',\dots,d_n') = d_1' \leq d_n = \max(d_1,\dots,d_n)  \quad \Rightarrow \quad d_1 - d_n < - \gamma(n)
$$
Because of this, the pigeonhole principle implies that there exists $k \in \{1,\dots,n-1\}$ such that:
\begin{equation}
\label{eqn:d ineq}
d_k - d_{k+1} < - \frac {\gamma(n)}{n-1} \leq - \gamma(k) - \gamma(n-k) - 2mn
\end{equation}

\begin{claim}
\label{claim:graph}

Only vertices of the form:
\begin{equation}
\label{eqn:sequence}
(s_1,\dots,s_n) \quad \text{with} \quad \begin{cases} |s_1 - d_1| \leq \gamma(k) \\  |s_k - d_k| \leq \gamma(k) \\ |s_{k+1} - d_{k+1}| \leq \gamma(n-k) \\ |s_n - d_n| \leq \gamma(n-k) \end{cases}
\end{equation}
can be reached by a path in $G$ starting from $(d_1,\dots,d_n)$. \\

\end{claim}

\noindent The Claim concludes the proof of the Lemma, as we assumed that $|d_1'-d_1| > \gamma(n) > \gamma(k)$, which means that the sequence $(d_1',\dots,d_n')$ is not among the \eqref{eqn:sequence}. \\

\begin{proof} \emph{of Claim \ref{claim:graph}:} We will prove the required statement by induction on the length of the path. Indeed, assume we have a path in $G$ of the form:
\begin{equation}
\label{eqn:path}
(d_1,\dots,d_n) \longrightarrow \dots \longrightarrow (s_1,\dots,s_n) \longrightarrow (t_1,\dots,t_n)
\end{equation}
and the induction hypothesis tells us that $(s_1,\dots,s_n)$, as well as all the vertices on the path \eqref{eqn:path} between $(d_1,\dots,d_n)$ and $(s_1,\dots,s_n)$, are of the form \eqref{eqn:sequence}. As:
\begin{equation}
\label{eqn:s ineq}
s_k - s_{k+1} = \underbrace{s_k - d_k}_{\leq \gamma(k)} + \underbrace{d_k - d_{k+1}}_{<-\gamma(k) - \gamma(n-k) - 2mn} + \underbrace{d_{k+1} - s_{k+1}}_{\leq \gamma(n-k)} < - 2mn
\end{equation}
the sentence after relation \eqref{eqn:necessary} implies that all edges emanating from $(s_1,\dots,s_k)$ correspond to permutations $\sigma$ that send $\{1,\dots,k\}$ to itself and $\{k+1,\dots,n\}$ to itself. However, the same is true for all intermediate vertices along the path \eqref{eqn:path}, and thus in getting from $(d_1,\dots,d_n)$ to $(t_1,\dots,t_n)$ one only uses permutations $\sigma$ that send $\{1,\dots,k\}$ to itself. The fact that $(t_1,\dots,t_n)$ is of the form \eqref{eqn:sequence} then follows from the induction hypothesis of the underlined claim on the previous page.

\end{proof}

\end{proof}

\begin{remark}
\label{rem:root}

Lemma \ref{lem:combi} is the type $A_{n-1}$ version of the following statement, which we invite the interested reader to prove in complete generality. For a finite type root system, fix a choice of positive and negative roots $\Delta = \Delta^+ \sqcup \Delta^-$. We will denote the weight lattice by $P$, and the cone of dominant weights by $P^+$ (the latter is a fundamental chamber for the action of the Weyl group $W$ on $P$). Fix a natural number $m$, and let $G$ be the graph with vertex set $P^+$ and edge set:
\begin{equation}
\label{eqn:dominant weights}
\lambda \quad \longrightarrow \quad \sigma(\lambda) + \sum_{\alpha \in \Delta^+ \cap \sigma(\Delta^-)} c_\alpha \cdot \alpha
\end{equation}
for any $\sigma \in W$ and any $c_\alpha \in \mathbb{Z}_{\leq m}$ (it is implied that the weight in the right-hand side of \eqref{eqn:dominant weights} should be dominant, in order for the right-hand side of \eqref{eqn:dominant weights} to define an edge set on $P^+$). Show that all the connected components of $G$ are finite. \\

\end{remark} 

\section{Twists of the shuffle product and Hopf algebras}
\label{sec:r}

\medskip

\subsection{} 
\label{sub:r}

As we have seen in Remark \ref{rem:modify}, using different line bundles from \eqref{eqn:line bundle} leads to multiplying the rational function $\zeta_{ij}$ of \eqref{eqn:def zeta} by $\pm$ a suitable monomial. The particular example we will consider in the present Section is:
\begin{equation}
\label{eqn:def zeta final}
\zeta'_{ij}(x) = \left(\frac {1-xq^{-1}}{1-x} \right)^{\delta_j^i} \prod_{e = \oij \in E} \left(\frac 1{t_e} - x \right) \prod_{e = \oji \in E} \left(1 - \frac {t_e}{qx} \right)
\end{equation}
We will consider $\CS' = \CS$ as an $\BF$-vector space, but make $\CS'$ into an algebra using the multiplication \eqref{eqn:shuf prod} with $\zeta'_{ij}$ instead of $\zeta_{ij}$. As we will see in the following Subsections, this has a minimal effect on our treatment of $\CS'$ as an \underline{algebra}, but it allows us to think of it as a \underline{bialgebra}. Formula-wise, this happens because:
\begin{align}
&\frac {\zeta'_{ij}(x)}{\zeta'_{ji}(x^{-1})} \Big|_{x = 0} = q^{\delta_j^i} \prod_{e = \oij \in E} \frac 1{t_e} \prod_{e = \oji \in E} \frac {t_e}q \label{eqn:constant term 1} \\
&\frac {\zeta'_{ij}(x)}{\zeta'_{ji}(x^{-1})} \Big|_{x = \infty} = \frac 1{q^{\delta_j^i}} \prod_{e = \oij \in E} \frac q{t_e} \prod_{e = \oji \in E} t_e  \label{eqn:constant term 2}
\end{align}
as opposed from the analogous ratios for the function $\zeta_{ij}$, which have zeroes/poles at $0$/$\infty$. The RHS of \eqref{eqn:constant term 1}--\eqref{eqn:constant term 2} can be construed as certain deformations (in the sense of the presence of the parameters $t_e$) of the usual $q$-Euler form of the quiver $Q$. \\

%We begin by remarking that $\CS''$ can be extended as in Definition \ref{def:extended shuffle} and doubled as in Subsection \ref{sub:r-matrix}, with all occurrences of $\zeta_{ij}$ and $\zeta'_{ij}$ from Section \ref{sec:r} replaced by $\zeta_{ij}''$. The reason one can do this is the following analogue of \eqref{eqn:constant term}:
%\begin{equation}
%\label{eqn:constant term alt}
%\frac {\zeta''_{ji}(x^{-1})}{\zeta''_{ij}(x)}\in q^{(\bs_i,\bs_j)} + x\BF[[x]]
%\end{equation}
%Thus, the functions $\zeta_{ij}''$ are more ``balanced" than $\zeta_{ij}$ or $\zeta'_{ij}$. \\

\begin{example}
\label{ex:loop}

When $Q$ has no loops or multiple edges, and we let $t_e = q^{\frac 12}, \ \forall e \in E$, the algebra $\CS'$ is isomorphic to the algebra $\overline{Sh}$ of \cite{E} (our $q$ is their $q^2$). In Subsection \ref{sub:km quiver}, we will use  Theorem \ref{thm:main} (or more precisely, Corollary \ref{cor:gen}, where we deal with the situation of specialized parameters) to show that this shuffle algebra is isomorphic to the positive half of the quantum loop group associated to $Q$. \\

\end{example}

\begin{example}
\label{ex:eha}

When $Q$ is the Jordan quiver (one vertex and one loop $e$), the algebra $\CS'$ is isomorphic to the spherical elliptic Hall algebra of \cite{BS}. More specifically, the following map from $\CS'$ to the shuffle algebra $\CA^+$ studied in \cite{Shuf}:
$$
R(z_1,\dots,z_n) \mapsto R(z_1,\dots,z_n) \prod_{1\leq i \neq j \leq n} \frac {1 - \frac {z_i}{z_j}}{\left(1 - \frac {z_i}{z_jq_1} \right)\left(1 - \frac {z_i}{z_jq_2} \right)}
$$
is an isomorphism (the parameters $q_1$ and $q_2$ of $\CA^+$ are identified with our $t_e$ and $\frac q{t_e}$). It was shown in \cite{Shuf} that $\CA^+$ is isomorphic to the spherical elliptic Hall algebra. \\

\end{example}

\subsection{} We will now show how to modify the contents of Section \ref{sec:proof} to obtain the analogue of Theorem \ref{thm:main} for the algebra $\CS'$ instead of $\CS$. In what follows, every time we say ``just like in Section \ref{sec:proof}", we mean ``just like in Section \ref{sec:proof}, with the rational function $\zeta_{ij}$ replaced by $\zeta'_{ij}$". Proposition \ref{prop:pairing} carries through just like in Section \ref{sec:proof}, and the first place where we need to make a substantial modification is in Definition \ref{def:proto}. Specifically, we now call a word:
$$
\left[ i_1^{(d_1)} \dots i_n^{(d_n)} \right]
$$
\textbf{non-increasing} if we have the following inequalities for all $1 \leq a < b \leq n$:
\begin{equation}
\label{eqn:non-increasing 2}
\begin{cases} d_a < d_b + \sum_{s=a}^{b-1} \#_{i_si_b} \\ \quad \text{or} \\ d_a = d_b + \sum_{s=a}^{b-1}  \#_{i_si_b} \text{ and } i_a \geq i_b \end{cases}
\end{equation}
where $\#_{ij}$ was defined in \eqref{eqn:def sharp}. \\ %The particular notion of non-increasing words as in \eqref{eqn:non-increasing 2} is necessary in order for Proposition \ref{prop:proto} to carry through just like in Section \ref{sec:proof}. \\

\begin{proof} \emph{of Proposition \ref{prop:proto} in the case at hand:} we will prove the analogue of \eqref{eqn:proto strong} by induction on $n$. Running the natural analogue of the proof of Proposition \ref{prop:proto}, every $e_v$ can be written as a linear combination of $e_w$'s with $w \geq v$ such that:
\begin{equation}
\label{eqn:words remark}
\text{if} \quad w = \left[i_1^{(d_1)} \dots i_n^{(d_n)} \right] \quad \text{then} \quad d_k \leq d_{k+1} + \#_{i_ki_{k+1}}, \ \forall k
\end{equation}
Since this property is weaker than \eqref{eqn:non-increasing 2}, some further explanation is in order. First of all, if a word $w$ as above satisfies $d_{k+1} - d_k > 2\max(\beta(k), \beta(n-k)) + 6n|E|$ for some $k \in \{1,\dots,n-1\}$ (where $\beta(n)$ denotes a natural number which ensures that the analogue of \eqref{eqn:proto strong} holds), then we can use the induction hypothesis to write:
$$
e_{i_1,d_1} * \dots * e_{i_k,d_k} \qquad \text{and} \qquad e_{i_{k+1},d_{k+1}} * \dots * e_{i_n,d_n}
$$
as linear combinations of non-increasing words in the sense of \eqref{eqn:non-increasing 2}. Moreover, the concatenations of the respective non-increasing words will also be non-increasing due to the large gap between $d_k$ and $d_{k+1}$. Therefore, we are left to contend with the finitely many (in each degree) words \eqref{eqn:words remark} where the $d_k$'s are all contained in an interval of some universally bounded length. The fact that there are finitely many such words is crucial, as it reduces our task to the following weaker fact. \\

\begin{claim}
\label{claim:word}

If $w$ is not a non-increasing word, then one can write $e_w$ as a linear combination of $e_y$'s with $y > w$. \\

\end{claim}

\begin{proof} \emph{of Claim \ref{claim:word}:} If the word $w$ is of the form \eqref{eqn:words remark}, then we cannot prove the Claim above just by applying a single quadratic relation \eqref{eqn:quadratic}. However, iterating formula \eqref{eqn:quadratic} with $\zeta$ replaced by $\zeta'$ implies that for any permutation $\sigma \in S(n)$:
\begin{multline}
\label{eqn:quadratic many}
e_{i_1}(z_1) * e_{i_2}(z_2) * \dots * e_{i_n}(z_n) \prod^{a < b}_{\sigma^{-1}(a) > \sigma^{-1}(b)} \zeta'_{i_b i_a} \left(\frac {z_b}{z_a} \right) = \\ = e_{i_{\sigma(1)}}(z_{\sigma(1)}) * e_{i_{\sigma(2)}}(z_{\sigma(2)}) * \dots * e_{i_{\sigma(n)}}(z_{\sigma(n)}) \prod_{\sigma(a) > \sigma(b)}^{a<b} \zeta'_{i_{\sigma(b)} i_{\sigma(a)}} \left(\frac {z_{\sigma(b)}}{z_{\sigma(a)}} \right) 
\end{multline}
We will assume that the word $w = [i_1^{(d_1)} \dots i_n^{(d_n)}]$ is ``almost" non-increasing, in the sense that \eqref{eqn:non-increasing 2} holds for all $(a,b) \neq (1,n)$, but the opposite holds for $(a,b) = (1,n)$:
\begin{equation}
\label{eqn:failure}
d_1 \geq d_n + \sum_{s=1}^{n-1} \#_{i_si_n} + \begin{cases} 1 &\text{if } i_1 \leq i_n \\ 0 &\text{if } i_1 > i_n \end{cases}
\end{equation}
(otherwise, we could just take a maximal sub-word of $w$ which is almost non-increasing in the sense above, and run the subsequent argument for the sub-word). It is easy to see that $i_n \neq i_a$ for all $a \in \{2,\dots,n-1\}$, otherwise \eqref{eqn:failure} would contradict \eqref{eqn:non-increasing 2} for  $(a,b) \neq (1,n)$. Then let us apply \eqref{eqn:quadratic many} for the permutation:
\begin{equation}
\label{eqn:permutation}
\sigma = \begin{pmatrix} 1 & 2 & \dots & n \\ n & 1 & \dots & n-1 \end{pmatrix}
\end{equation}
and extract the coefficient of $\prod_{a=1}^n z_a^{- d_a - \sum^{a>s}_{\sigma^{-1}(a) < \sigma^{-1}(s)} \#_{\overrightarrow{i_si_a}} + \sum^{a<t}_{\sigma^{-1}(a) > \sigma^{-1}(t)} \#_{\overrightarrow{i_ai_t}}}$:
\begin{multline}
\label{eqn:equality of words}
e_{i_1,d_1} * e_{i_2,d_2} * \dots * e_{i_n,d_n} + \text{larger words} = \\ = (-1)^{\delta_{i_n}^{i_1}} \cdot e_{i_{n},d'_{n}} * e_{i_{1},d'_{1}} * \dots * e_{i_{n-1},d'_{n-1}} + \text{larger words} 
\end{multline}
where the phrase ``larger words" immediately following $e_w$ is shorthand for ``a linear combination of $e_y$'s with $y>w$", and we set:
\begin{align*}
&d'_{a} = d_{a} - \#_{i_ai_n} - \delta_a^1 \delta_{i_n}^{i_1}, \ \forall a < n \\
&d'_{n} = d_{n} + \sum_{s=1}^{n-1} \#_{i_si_n} + \delta_{i_n}^{i_1}
\end{align*}
(the Kronecker $\delta$ functions appear because the rational function $\zeta_{i_1i_n}$ has a linear factor in the denominator, which needs to be cleared from \eqref{eqn:quadratic many}). By \eqref{eqn:failure}, we have:
$$
\left[ i_1^{(d_1)}  i_2^{(d_2)} \dots i_n^{(d_n)} \right] \leq \left[ i_n^{(d_n')} i_1^{(d_1')} \dots i_{n-1}^{(d_{n-1}')} \right]
$$
with equality only if $i_1 = i_n$. Thus, \eqref{eqn:equality of words} allows us to write $e_w$ as a linear combination of $e_y$'s for various words $y > w$, as we needed to show. 

\end{proof} \end{proof}

\noindent The notion of standard words is defined just like in Section \ref{sec:proof}, and the next place we encounter a difference is in \eqref{eqn:ratio zero}. In the case at hand, the ratio of zeta functions therein is actually regular at $0$. Therefore, the analogue of \eqref{eqn:non-zero} tells us that for all non-increasing words $v$ and $w$ of the form \eqref{eqn:form}, we have $\langle e_v, f_w \rangle \neq 0$ only if:
$$
(k_1,\dots,k_n) = (d_{\sigma^{-1}(1)}, \dots, d_{\sigma^{-1}(n)}) + \sum^{a < b}_{\sigma^{-1}(a) > \sigma^{-1}(b)} c_{a,b} \cdot \underbrace{(0,\dots,1,\dots,-1,\dots,0)}_{1\text{ on position }a, -1 \text{ on position }b}
$$
for some $\sigma \in S(n)$ such that $i_a = j_{\sigma(a)}, \ \forall a$ and some $\{c_{a,b} \geq 0\}^{a < b}_{\sigma^{-1}(a) > \sigma^{-1}(b)}$.  \\

%However, because $v$ and $w$ are non-increasing words, the sequences $(d_1,\dots,d_n)$ which appears in the expression above must satisfy $d_a \leq d_{a+1} + \#_{i_ai_{a+1}}$ for all $a$, and analogously for the sequence $(k_1,\dots,k_n)$. \\

\noindent The preceding discussion means that the graph $G$ defined in Subsection \ref{sub:graph} should be replaced by the graph $G'$ with vertices:
$$
(d_1,\dots,d_n) \quad \text{such that} \quad d_a \leq d_{a+1} + m, \quad \forall a \in \{1,\dots,n-1\}
$$
(where $m = 2|E|$) and edges as in \eqref{eqn:edge} only for those $c_{a,b} \geq 0$ for all $a,b$. However, it is easy to see that $G$ and $G'$ are isomorphic graphs, upon the one-to-one correspondence of vertices:
$$
(d_1,\dots,d_n) \in G \quad \leadsto \quad \left(d_1 + \frac {m(n-1)}2, d_2 + \frac {m(n-3)}2, \dots, d_n - \frac {m(n-1)}2 \right) \in G'
$$
(this statement is elementary, and left as an exercise to the interested reader; it uses the fact that for any permutation $\sigma \in S(n)$ and any $a \in \{1,\dots,n\}$, the number of those $s<a$ such that $\sigma(s) > \sigma(a)$ minus the number of those $t > a$ such that $\sigma(t) < \sigma(a)$ is equal to $a-\sigma(a)$). This means that Lemma \ref{lem:combi} applies to $G'$, which we may conclude to have finite connected components. Then the contents of Subsections \ref{sub:direct sum} and \ref{sub:proof} go through as stated, thus leading to a proof of the following. \\

\begin{theorem}
\label{thm:alt main}

The algebra $\CS'$ coincides with its subalgebra $\oCS'$ generated by $\{e_{i,d}\}^{i\in I}_{d\in \BZ}$. \\

\end{theorem}

\subsection{} 

We will now show how to adapt the notion of leading words from Subsection \ref{sub:pbw} to the present setup; this will also serve as additional motivation for the notion of non-increasing words from \eqref{eqn:non-increasing 2}. Consider any ordered monomial:
\begin{equation}
\label{eqn:monomial}
z_{i_1 \bullet_1}^{k_1} \dots z_{i_n \bullet_n}^{k_n}
\end{equation}
(where $i_1,\dots,i_n \in I$, $k_1,\dots, k_n \in \BZ$, $\bullet_1,\dots, \bullet_n \geq 0$ are such that $\bullet_a \neq \bullet_b$ if $a \neq b$ and $i_a = i_b$). The \textbf{associated word} of the ordered monomial \eqref{eqn:monomial} is:
\begin{equation}
\label{eqn:lead word}
\left[i_1^{(d_1)} \dots i_n^{(d_n)} \right] \quad \text{where} \quad d_a = k_a - \sum_{s < a} \#_{\overrightarrow{i_ai_s}} + \sum_{a < t} \#_{\overrightarrow{i_ti_a}}, \quad \forall a \in \{1,\dots,n\}
\end{equation}
The lexicographically largest of the associated words of various orderings of a given monomial $\mu$ will be called the \textbf{leading word} of $\mu$. It is easy to see that two orderings of a given monomial give rise to the same associated word if and only if they correspond to a permutation $\sigma \in S(n)$ such that $i_a = i_{\sigma(a)}$ and $k_a = k_{\sigma(a)}$ for all $a \in \{1,\dots,n\}$, i.e. the two orderings only differ in the indices $\bullet_1,\dots,\bullet_n$. This implies that the leading word of $\mu$ only depends on $\text{Sym }\mu$. \\

\begin{lemma}
\label{lem:leading word}

Among all the associated words of a monomial \eqref{eqn:monomial}, the leading word is the only one which is non-increasing in the sense of \eqref{eqn:non-increasing 2}. \\

\end{lemma}

\begin{proof} Let us first show that the leading word is non-increasing. Assume it arises from an ordering as in \eqref{eqn:monomial}. For any $1 \leq a < b \leq n$, consider the permutation:
\begin{equation}
\label{eqn:permutation 2}
\sigma = \begin{pmatrix} 1 & \dots & a-1 & a & a+1 & \dots & b & b+1 & \dots & n \\ 1 & \dots & a-1 & b & a & \dots & b-1 & b+1 & \dots & n \end{pmatrix}
\end{equation}
The very definition of the leading word implies that:
\begin{equation}
\label{eqn:ineq leading word}
\left[i_1^{(d_1)} \dots i_n^{(d_n)}\right] \geq \left[i_{\sigma(1)}^{(d_{\sigma(1)}')} \dots i_{\sigma(n)}^{(d'_{\sigma(n)})}\right]
\end{equation}
where for any $c \in \{1,\dots,n\}$, we have:
\begin{align}
&d_c = k_c - \sum_{s < c} \#_{\overrightarrow{i_ci_s}} + \sum_{c < t} \#_{\overrightarrow{i_ti_c}} \label{eqn:d8} \\
&d'_{\sigma(c)} = k_{\sigma(c)} - \sum_{s < c} \#_{\overrightarrow{i_{\sigma(c)}i_{\sigma(s)}}} + \sum_{c < t} \#_{\overrightarrow{i_{\sigma(t)}i_{\sigma(c)}}} \label{eqn:d9}
\end{align}
Eliminating the $k$'s from the formulas above implies (recall that $\#_{ij} = \#_{\oij}+\#_{\oji}$):
\begin{align*}
d'_{\sigma(c)} - d_{\sigma(c)} &= \sum_{c < t} \#_{\overrightarrow{i_{\sigma(t)}i_{\sigma(c)}}} - \sum_{\sigma(c) < \sigma(t)} \#_{\overrightarrow{i_{\sigma(t)}i_{\sigma(c)}}} - \sum_{s < c}  \#_{\overrightarrow{i_{\sigma(c)}i_{\sigma(s)}}} + \sum_{\sigma(s) < \sigma(c)} \#_{\overrightarrow{i_{\sigma(c)}i_{\sigma(s)}}} \\
&= \sum^{t>c}_{\sigma(t) < \sigma(c)} \#_{i_{\sigma(t)}i_{\sigma(c)}} - \sum^{s<c}_{\sigma(s) > \sigma(c)} \#_{i_{\sigma(c)}i_{\sigma(s)}}
\end{align*}
For $\sigma$ as in \eqref{eqn:permutation 2}, the formula above implies $d'_1 = d_1, \dots, d'_{a-1} = d_{a-1}$ and: 
\begin{equation}
\label{eqn:eq leading word}
d'_b - d_b = \sum_{s=a}^{b-1} \#_{i_si_b}
\end{equation}
The only way  \eqref{eqn:ineq leading word} can be satisfied is if $d_b' = d_{\sigma(a)}' > d_a$ or if $d_b' = d_{\sigma(a)}' = d_a$ and $i_b = i_{\sigma(a)} \leq i_a$. By \eqref{eqn:eq leading word}, this is precisely equivalent to condition \eqref{eqn:non-increasing 2}. \\

\noindent To show the ``only" part of Lemma \ref{lem:leading word}, we must show that two orderings of a given monomial \eqref{eqn:monomial} cannot give rise to distinct non-increasing associated words. Thus, assume for the purpose of contradiction that the aforementioned associated words:
\begin{equation}
\label{eqn:two leading words}
\left[i_1^{(d_1)} \dots i_n^{(d_n)}\right] > \left[i_{\sigma(1)}^{(d_{\sigma(1)}')} \dots i_{\sigma(n)}^{(d'_{\sigma(n)})}\right]
\end{equation}
are both non-increasing. Choose $a \in \{1,\dots,n\}$ such that the first $a-1$ letters of the words above match, but the $a$-th letter of the word on the left is greater than the $a$-th letter of the word on the right. Thus, for all $c <a$ we have $i_c = i_{\sigma(c)}$ and $d_c = d'_{\sigma(c)}$; a straightforward application of \eqref{eqn:d8} and \eqref{eqn:d9} implies that also $k_c = k_{\sigma(c)}$. Since reordering variables of a monomial \eqref{eqn:monomial} with the same $i_c$ and $k_c$ does not change the associated word of the monomial, we may assume for simplicity that $\sigma(c) = c$ for all $c < a$. However, our hypothesis on the number $a$ shows that:
\begin{equation}
\label{eqn:ineq 1}
d_a \leq d'_{\sigma(a)}
\end{equation}
with equality only if $i_a > i_{\sigma(a)}$. Let $b = \sigma^{-1}(a) \Leftrightarrow \sigma(b) = a$. Because $\sigma(c) = c$ for all $c<a$ (and because \eqref{eqn:ineq 1} precludes $a = b$), we must have $a < b$. Property \eqref{eqn:non-increasing 2} applied to the word in the right-hand side of \eqref{eqn:two leading words} reads:
\begin{equation}
\label{eqn:ineq 2}
d'_{\sigma(a)} \leq d'_{a} + \sum_{s=a}^{b-1} \#_{i_{\sigma(s)}i_{a}}  
\end{equation}
with equality only if $i_{\sigma(a)} \geq i_{a}$. Combining the two inequalities above implies:
\begin{equation}
\label{eqn:ineq 3}
d_a < d_a' + \sum_{s=a}^{b-1} \#_{i_{\sigma(s)}i_{a}}  
\end{equation}
However, the formula immediately preceding \eqref{eqn:eq leading word} (for $c \leadsto b = \sigma^{-1}(a)$) reads:
$$
d'_{a} - d_{a} = \sum^{t>b}_{\sigma(t) < a} \#_{i_{\sigma(t)}i_{a}} - \sum^{s<b}_{\sigma(s) > a} \#_{i_{a}i_{\sigma(s)}}
$$
which combined with \eqref{eqn:ineq 3} implies:
$$
 -\sum_{s=a}^{b-1} \#_{i_{\sigma(s)}i_{a}}  < \sum^{t>b}_{\sigma(t) < a} \#_{i_{\sigma(t)}i_{a}} - \sum^{s<b}_{\sigma(s) > a} \#_{i_{a}i_{\sigma(s)}} 
$$
Because $\sigma(c) = c$ for all $c < a$, the two sides of the inequality above are actually equal to each other (the first sum in the right-hand side is vacuous, and the second sum in the right-hand side runs over the same indexing set as the sum in the left-hand side), thus giving us the required contradiction.

\end{proof}

\noindent The leading word of an element $R \in \CS$ is defined as the lexicographically largest of the leading words of all of its constituent monomials. With this in mind, we leave the following analogue of \eqref{eqn:standard eq} as an exercise to the interested reader:
\begin{equation}
\label{eqn:standard eq 2}
\Big \langle R, f_w \Big \rangle \text{ is } \begin{cases} \neq 0 &\text{if }w = v \\ = 0 &\text{if }w > v \end{cases}
\end{equation}
where $v$ denotes the leading word of $R$. Indeed, we may compute the LHS by applying formula \eqref{eqn:pairing formula} with $R$ replaced by the symmetrization of the monomial \eqref{eqn:monomial} and $\zeta_{ij}(x) \in \BF[[x]]^\times$ replaced by $\zeta'_{ij}(x) \in x^{-\#_{\oji}}\BF[[x]]^\times$. That the resulting expression equals the RHS of \eqref{eqn:standard eq 2} is a straightforward consequence of Lemma \ref{lem:leading word}. \\ 

\subsection{} 
\label{sub:cop plus}

Because of \eqref{eqn:constant term 1}--\eqref{eqn:constant term 2}, we may make $\CS'$ into a Hopf algebra (various incarnations of this process were carried out in numerous papers, most notable for our situation being \cite{Shuf, SV, YZ}). As is common in the theory of quantum loop groups, we must first extend and double the algebra $\CS'$, and we will now recall the details. \\

\begin{definition}
\label{def:extended shuffle}

Consider the extended algebra:
\begin{equation}
\label{eqn:extended shuffle}
{\CS'}^{\geq} = \CS' \bigotimes_{\BF} \BF \left[h^+_{i,d}\right]_{i \in I, d \geq 0}
\end{equation}
where the multiplication is governed by the following relation for all $i,j \in I$:
\begin{equation}
\label{eqn:rel ext shuffle}
R(\dots ,z_{ia}, \dots) h^+_j(w) = h^+_j(w) R(\dots ,z_{ia}, \dots) \prod^{i\in I}_{1\leq a \leq n_i} \frac {\zeta'_{ij} \left( \frac {z_{ia}}{w} \right)}{\zeta'_{ji} \left( \frac {w}{z_{ia}} \right)}
\end{equation}
where the RHS is defined by expanding as a power series in $|z_{ia}| \ll |w|$, and:
$$
h_j^+(w) = \sum_{d = 0}^{\infty} \frac {h_{j,d}^+}{w^d}
$$

\end{definition}

\medskip

%\noindent By taking the $w^0$ coefficient of \eqref{eqn:rel ext shuffle}, we obtain the following relation:
%$$
%h_i R h_i^{-1} = \dots R
%$$
%for any $R \in \CS_{\bn}$ and any $i \in I$. \\

\noindent The following is a straightforward result, which we leave as an exercise to the interested reader (cf. \cite[eqn. (4.13), (4.14)]{Thesis}; alternatively, the proof presented in \cite{Shuf} for the particular case of the Jordan quiver carries through almost word-for-word): \\

%\begin{equation}
%\label{eqn:cop big quantum}
%\Delta (e_i(z)) = e_i(z) \otimes 1 + h_i(z) \otimes e_i(z)
%\end{equation}
%give rise to a (topological) coproduct on the algebra $U^{\emph{ext}}$

\begin{proposition}
\label{prop:coproduct big}

The assignments $\Delta(h^+_i(z)) = h^+_i(z) \otimes h^+_i(z)$ and:
\begin{equation}
\label{eqn:cop big shuffle}
\Delta (R(\dots,z_{i1},\dots,z_{in_i},\dots)) = 
\end{equation}
$$
= \sum_{\{k_i \in \{0,\dots,n_i\}\}_{i \in I}} \frac {\left[ \prod^{j \in I}_{k_j < b \leq n_j} h^+_j(z_{jb}) \right] \cdot R(\dots, z_{i1},\dots , z_{ik_i} \otimes z_{i,k_i+1}, \dots, z_{in_i},\dots)}{\prod^{i \in I}_{1\leq a \leq k_i} \prod^{j \in I}_{k_j < b \leq n_j} \zeta'_{ji} \left( \frac {z_{jb}}{z_{ia}} \right)}
$$
give rise to a (topological) coproduct on the algebra ${\CS'}^\geq$. To make sense of the right hand side of \eqref{eqn:cop big shuffle}, we expand the denominator as a power series in the range $|z_{ia}| \ll |z_{jb}|$, and place all the powers of $z_{ia}$ to the left of the $\otimes$ sign and all the powers of $z_{jb}$ to the right of the $\otimes$ sign (for all $i,j \in I$, $1 \leq a \leq k_i$, $k_j < b \leq n_j$). \\

\end{proposition}

\noindent The coproduct \eqref{eqn:cop big shuffle} is multiplicative, hence makes ${\CS'}^{\geq}$ into a bialgebra (the counit annihilates all $R \in \CS'_{\bn}$ for $\bn \neq 0$, and all $h^+_{i,d}$ with $d > 0$). It is straightforward to write the antipode that makes \eqref{eqn:extended shuffle} into a Hopf algebra, but we will not need it. \\

%\begin{proof} \emph{of Proposition \ref{prop:subword} (continued):} let $v = [i_1^{(d_1)} \dots i_n^{(d_n)}]$ be a standard word. If $R$ is an element of the shuffle algebra of the form \eqref{eqn:leading word}, formula \eqref{eqn:cop big shuffle} implies:
%$$
%\Delta(R) = \sum_{l=0}^n \text{Sym} \Big[ z_{i_1a_1}^{d_1} \dots z_{i_la_l}^{d_l} \Big] \otimes \text{Sym} \Big[ z_{i_{l+1} a_{l+1}}^{d_{l+1}} \dots z_{i_na_n}^{d_n} \Big] + \text{lower order monomials}
%$$
%where ``lower order monomials" refers to those tensors of the form:
%$$
%\prod_{j,d} h^+_{j,d} \cdot z_{j_1b_1}^{k_1} \dots z_{j_lb_l}^{k_l} \otimes  z_{j_{l+1} b_{l+1}}^{k_{l+1}} \dots z_{j_nb_n}^{k_n} 
%$$
%where either $d > 0$ for at least one $h^+_{j,d}$ in the product, or: 
%$$
%\left[j_1^{(k_1)} \dots j_l^{(k_l)}\right]  < \left[i_1^{(d_1)} \dots i_l^{(d_l)}\right] =: v_{\text{pref}}
%$$
%or: 
%$$
%\left[j_1^{(k_1)} \dots j_l^{(k_l)}\right]  = \left[i_1^{(d_1)} \dots i_l^{(d_l)}\right] \quad \text{and} \quad \left[j_{l+1}^{(k_{l+1})} \dots j_n^{(k_n)}\right]  < \left[i_{l+1}^{(d_{l+1})} \dots i_n^{(d_n)}\right] =: v_{\text{suff}}
%$$
%Since the tensor factors of $\Delta(R)$ are elements of $\CS^{\geq}$, this implies that the prefix $v_{\text{pref}}$ and the suffix $v_{\text{suff}}$ of $v$ are leading words. As we already showed in Subsection \ref{sub:pbw}, this implies that they are also standard words. Since any subword of $v$ is a prefix of a suffix of $v$, this concludes the proof. 
%\end{proof}

\subsection{} 
\label{sub:cop minus}

By analogy with the previous Subsection, we let:
$$
{\CS'}^{\leq} = {\CS'}^{\op} \bigotimes_{\BF} \BF \left[h^-_{i,d} \right]_{i \in I, d \geq 0}
$$
where the multiplication is governed by the following relation for all $i,j \in I$:
$$
R(\dots ,z_{ia}, \dots) h^-_j(w) = h^-_j(w) R(\dots ,z_{ia}, \dots) \prod^{i\in I}_{1\leq a \leq n_i} \frac {\zeta'_{ji} \left( \frac {w}{z_{ia}} \right)}{\zeta_{ij}' \left( \frac {z_{ia}}{w} \right)}
$$
where the RHS is defined by expanding as a power series in $|z_{ia}| \gg |w|$, and:
$$
h_j^-(w) = \sum_{d = 0}^{\infty} h_{j,d}^- w^d
$$
The assignment $\Delta(h^-_i(z)) = h^-_i(z) \otimes h^-_i(z)$ and:
$$
\Delta (R(\dots,z_{i1},\dots,z_{in_i},\dots)) = 
$$
$$
= \sum_{\{k_i \in \{0,\dots,n_i\}\}_{i \in I}} \frac {R(\dots, z_{i1},\dots , z_{ik_i} \otimes z_{i,k_i+1}, \dots, z_{in_i},\dots) \cdot \left[ \prod^{i \in I}_{1 \leq a \leq k_i} h^-_i(z_{ia}) \right]}{\prod^{i \in I}_{1\leq a \leq k_i} \prod^{i \in I}_{k_j < b \leq n_j} \zeta'_{ij} \left( \frac {z_{ia}}{z_{jb}} \right)}
$$
give rise to a (topological) coproduct on the algebra ${\CS'}^\leq$. To make sense of the right hand side of the expression above, we expand the denominator as a power series in the range $|z_{ia}| \ll |z_{jb}|$ for all $i,j \in I$ and all $1 \leq a \leq k_i$ and $k_j < b \leq n_j$. \\

\subsection{} By definition, a bialgebra pairing:
\begin{equation}
\label{eqn:ext pairing}
{\CS'}^{\geq} \otimes {\CS'}^{\leq} \xrightarrow{\langle \cdot , \cdot \rangle} \BF
\end{equation}
is a $\BF$-linear pairing which satisfies the following properties:
\begin{align}
&\Big \langle a,b_1b_2 \Big \rangle = \Big \langle \Delta(a), b_1 \otimes b_2 \Big \rangle \label{eqn:bialg 1} \\
&\Big \langle a_1a_2,b \Big \rangle = \Big \langle a_1 \otimes a_2, \Delta^{\op}(b) \Big \rangle \label{eqn:bialg 2}
\end{align}
for all $a,a_1,a_2 \in {\CS'}^\geq$ and $b,b_1,b_2 \in {\CS'}^{\leq}$. In other words, the pairing is such that the dual of the product is the (opposite) coproduct, and vice versa. There exists also the stronger notion of Hopf pairing, which additionally satisfies the property:
\begin{equation}
\label{eqn:hopf alg}
\Big \langle S(a), S(b) \Big \rangle = \Big \langle a,b \Big \rangle
\end{equation}
for all $a \in {\CS'}^\geq$ and $b \in {\CS'}^{\leq}$ (we will not need \eqref{eqn:hopf alg} in the present paper). \\

\begin{proposition}
\label{prop:bialgebra pairing}

There is a unique bialgebra pairing \eqref{eqn:ext pairing} which satisfies:
\begin{equation}
\label{eqn:pair h's}
\Big \langle h_i^+(z), h_j^-(w) \Big \rangle = \frac {\zeta_{ij}' \left(\frac zw \right)}{\zeta_{ji}' \left(\frac wz \right)} 
\end{equation}
(the RHS is expanded as $|z| \gg |w|$) and whose restriction to:
\begin{equation}
\label{eqn:pairing prime}
\CS' \otimes {\CS'}^{\emph{op}} \xrightarrow{\langle \cdot , \cdot \rangle} \BF
\end{equation}
is given by formula \eqref{eqn:pairing formula} with $\zeta'_{ij}$ instead of $\zeta_{ij}$. \\

\end{proposition}

\noindent The Proposition above is proved with minor modifications in \cite[Exercise IV.2.]{Thesis}; alternatively, it is proved almost word-for-word as its particular case when $Q$ is the Jordan quiver, which the interested reader may find in \cite{Shuf}. Let us note that formula \eqref{eqn:pairing formula} (respectively \eqref{eqn:symmetric pairing}) with $\zeta_{ij}$ replaced by $\zeta_{ij}'$ manifestly shows that the pairing \eqref{eqn:pairing} satisfies property \eqref{eqn:bialg 1} when $a \in \CS'$ and $b_1,b_2 \in {\CS'}^{\op}$ (respectively property \eqref{eqn:bialg 2} when $a_1,a_2 \in \CS'$ and $b \in {\CS'}^{\op}$). \\

\subsection{}
\label{sub:r-matrix}

Given the bialgebra pairing \eqref{eqn:ext pairing}, we may define the Drinfeld double as:
$$
\DD \CS' = {\CS'}^{\geq} \otimes {\CS'}^{\leq}
$$
where the multiplication in the algebra above is governed by the relation:
$$
a_1 b_1 \Big \langle a_2,b_2 \Big \rangle = b_2 a_2 \Big \langle a_1, b_1 \Big \rangle
$$
for any $a \in {\CS'}^{\geq} \otimes 1 \subset \DD \CS'$ and any $b \in 1 \otimes {\CS'}^{\leq} \subset \DD \CS'$. In the formula above, we use Sweedler notation for the coproduct: $\Delta(a) = a_1 \otimes a_2$ and $\Delta(b) = b_1 \otimes b_2$, with an implied summation sign. Drinfeld doubles such as $\DD \CS'$ are endowed with an important distinguished element:
$$
\CR \in \DD \CS' \ \widehat{\otimes} \ \DD \CS'
$$
called a universal $R$-matrix (the completion is necessary because our coproduct is topological). As is well-known in the theory of quantum groups, we have:\footnote{For a survey of the formula \eqref{eqn:univ r} in the particular case of the Jordan quiver, we refer the reader to \cite{R-matrix}, where we recall the standard difficulties in properly defining the product in \eqref{eqn:univ r}.}
\begin{equation}
\label{eqn:univ r}
\CR = \Big[\text{an expression involving the }h^{\pm}_{i,d} \Big] \cdot \CR'
\end{equation}
where $\CR'$ is the canonical tensor of the pairing \eqref{eqn:pairing prime}:
\begin{equation}
\label{eqn:univ r'}
\CR' = \sum_{w \text{ standard}} e^w \otimes f_w
\end{equation}
Recall from Theorem \ref{thm:main} and \eqref{eqn:direct sum 1} that $\CS'$ and ${\CS'}^{\op}$ decompose as direct sums of mutually orthogonal finite-dimensional pieces indexed by the connected components $H \subset G$ of Subsection \ref{sub:direct sum}. Therefore, $\CR'$ is a sum of finite contributions indexed by the various $H \subset G$, and these contributions can be computed explicitly (albeit not in a very useful way, in the author's opinion) using formula \eqref{eqn:pair e and f}. \\ 

\section{The case of non-generic parameters}
\label{sec:tori}

\medskip

\subsection{} 
\label{sub:works}

In the present Section, we will replace $\BF$ by an arbitrary field $\BK$ endowed with non-zero elements $q$ and $\{t_e\}_{e \in E}$. Since there might be non-trivial algebraic relations between the elements $q, t_e$, the goal is now to define a shuffle algebra:
\begin{equation}
\label{eqn:restricted shuffle}
{_\BK\CS} \subset {_\BK\CV} := \bigoplus_{\bn = (n_i)_{i \in I} \in \nn} \BK\left[z_{i1}^{\pm 1}, \dots, z_{in_i}^{\pm 1} \right]^\sym_{i \in I}
\end{equation}
to which the analogue of Theorem \ref{thm:main} applies. We make the following choice: \\

\begin{definition}
\label{def:special shuffle}

Let the vector subspace ${_\BK\CS}$ of \eqref{eqn:restricted shuffle} consist precisely of those symmetric Laurent polynomials $R(\dots, z_{ia}, \dots)$ such that:
\begin{equation}
\label{eqn:divisible 2}
\prod_{e = \oij} \left( z_{ia} - \frac {qz_{j1}}{t_e} \right) \prod_{e = \oji} (z_{ia} - t_{e} z_{j1}) \quad \text{divides} \quad R \Big|_{z_{j2} = q z_{j1}}
\end{equation}
for all $i,j \in I$ and all $a$ such that $(i,a) \notin \{(j,1), (j,2)\}$. \\

\end{definition}

\noindent If the scalars $\left \{ \frac q{t_e}, t_{e'} \right\}_{e = \oij, e' = \oji}$ are all distinct, \eqref{eqn:divisible 2} is equivalent to \eqref{eqn:wheel}. \\

\begin{proposition}
\label{prop:restricted subalg}

$_\BK\CS$ is a subalgebra of $_\BK\CV$. \\

\end{proposition}

\begin{proof} We need to show that if symmetric Laurent polynomials $R$ and $R'$ (of degrees $\bn$ and $\bn'$, respectively) satisfy \eqref{eqn:divisible 2}, then so does $R*R'$. Formula \eqref{eqn:shuf prod} reads:
\begin{equation}
\label{eqn:restricted shuf prod}
(R*R')(\dots, z_{i1}, \dots, z_{i,n_i+n'_i}, \dots) = \sum_{\{\sigma_i \in S(n_i+n_i')\}_{i \in I}} 
\end{equation}
$$
\frac {R(z_{i,\sigma_i(1)}, \dots, z_{i,\sigma_i(n_i)})  R'(z_{i,\sigma_i(n_i+1)}, \dots, z_{i,\sigma_i(n_i+n_i')})}{\prod_{i \in I} n_i! \prod_{i \in I} n_i'!} \mathop{\prod^{i,i' \in I}_{1\leq a \leq n_i}}_{n_{i'} < a' \leq n_{i'}+n'_{i'}} \zeta_{ii'} \left( \frac {z_{i,\sigma_i(a)}}{z_{i',\sigma_{i'}(a')}} \right)
$$
Let us fix $i,j \in I$ and specialize $z_{j2} = q z_{j1}$ in the formula above. We will show that for any collection of permutations $\{\sigma_i\}_{i\in I}$, the second line of \eqref{eqn:restricted shuf prod} is divisible by:
\begin{equation}
\label{eqn:the factor}
\prod_{e = \oij} \left( z_{ia} - \frac {qz_{j1}}{t_e} \right) \prod_{e = \oji} (z_{ia} - t_{e} z_{j1}) 
\end{equation}
for any $a \in \{1,\dots,n_i+n_i'\}$ (and $a \neq 1,2$ if $i = j$), according to the following cases: \\
 
\begin{itemize}[leftmargin=*]

\item if $\sigma_j^{-1}(1) , \sigma_j^{-1}(2) \leq n_j$, then the factor \eqref{eqn:the factor} for $\sigma_i^{-1}(a) \leq n_i$ arises because $R$ satisfies \eqref{eqn:divisible 2}, and for $\sigma^{-1}_i(a) > n_i$ because $\zeta_{ji}(z_{j1}/z_{ia})$ appears in \eqref{eqn:restricted shuf prod}; \\

\item if $\sigma_j^{-1}(1) , \sigma_j^{-1}(2) > n_j$, then the factor \eqref{eqn:the factor} for $\sigma_i^{-1}(a) > n_i$ arises because $R'$ satisfies \eqref{eqn:divisible 2}, and for $\sigma^{-1}_i(a) \leq n_i$ because $\zeta_{ij}(z_{ia}/z_{j2})$ appears in \eqref{eqn:restricted shuf prod}; \\

\item if $\sigma_j^{-1}(1) \leq n_j < \sigma_j^{-1}(2)$, then the factor \eqref{eqn:the factor} for $\sigma_i^{-1}(a) > n_i$ arises because $\zeta_{ji}(z_{j1}/z_{ia})$ appears in \eqref{eqn:restricted shuf prod}, and for $\sigma^{-1}_i(a) \leq n_i$ because $\zeta_{ij}(z_{ia}/z_{j2})$ appears. \\

\end{itemize}

\noindent The only case not covered by the preceding analysis is when $\sigma_j^{-1}(1) > n_j \geq \sigma_j^{-1}(2)$, when the appearance of the factor $\zeta_{jj}(z_{j2}/z_{j1}) = 0$ implies that the second line of \eqref{eqn:restricted shuf prod} vanishes altogether. We thus conclude that all summands in the right-hand side of \eqref{eqn:restricted shuf prod} are divisible by \eqref{eqn:the factor} for all $a \in \{1,\dots,n_i+n_i'\}$, as we needed to show.

\end{proof}

\subsection{}

As one goes through Section \ref{sec:proof}, one notes that the only place where we invoked the wheel conditions that determine the subalgebra $\CS \subset \CV$ was in the proof of Proposition \ref{prop:pairing}. Then let us henceforth make: \\

\noindent \textbf{Assumption \begin{otherlanguage*}{russian}Ъ\end{otherlanguage*}}: there exists a field homomorphism:
$$
\rho : \BK \rightarrow \BC
$$
for which $|\rho(q)| < |\rho(t_e)| < 1$ for all $e \in E$. \\

\begin{proposition}
\label{prop:restricted pairing}

Under Assumption \begin{otherlanguage*}{russian}Ъ\end{otherlanguage*}, Proposition \ref{prop:pairing} holds with:
$$
\CS, \oCS, \BF \quad \text{replaced by} \quad {_\BK\CS}, {_\BK\oCS}, \BK
$$
where ${_\BK\oCS} \subseteq {_\BK\CS}$ denotes the subalgebra generated by $\{z_{i1}^d\}_{i \in I, d \in \BZ}$. \\

\end{proposition}

\begin{proof} The only use of the wheel conditions in the proof of Proposition \ref{prop:pairing} was to ensure that for any $R \in {_{\BK}\CS}$ and any $i,j \in I$, $k \geq 1$, the rational function:
$$
\frac {R(\dots,z_{ia},\dots) |_{z_{jk} = x q^{k-1}, \dots, z_{j2} = xq, z_{j1} = x}}{\prod_{(i,a) \notin \{(j,1),\dots,(j,k)\}} \zeta_{ij} \left( \frac {z_{ia}}{x} \right) \dots \zeta_{ij} \left( \frac {z_{ia}}{xq^{k-1}} \right)}
$$
has no poles of the form $z_{ia} = x c$ with $|c| < 1$, other than $z_{ia} = xq^k$ (and the latter only if $i = j$), for all $j \in I$ and all $k \in \{1,\dots,n_j\}$. Looking back to the analysis in the three bullets in the proof of Claim \ref{claim:residues}, this amounts to ensuring that:
\begin{equation}
\label{eqn:divisible 1}
\prod_{s=1}^{k-1} \left[  \prod_{e = \oij} \left( z_{ia} - \frac {q^sx}{t_e} \right) \prod_{e' = \oji} (z_{ia} - q^{s-1}t_{e'} x) \right] \ \ \text{divides} \ \ R \Big|_{z_{jk} = xq^{k-1}, \dots,z_{j1} = x}
\end{equation}
for all $i,j \in I$ and all $1 \leq a \leq n_i$, $1 \leq k \leq n_j$ with  $(i,a) \notin \{(j,1),\dots,(j,k)\}$. However, because of  Assumption \begin{otherlanguage*}{russian}Ъ\end{otherlanguage*}, the multi-sets of scalars: 
$$
\left \{\frac {q^s}{t_e}, q^{s-1}t_{e} \right\}_{e \in E}
$$
are disjoint for different integers $s$. Thus, property \eqref{eqn:divisible 1} boils down to \eqref{eqn:divisible 2}. 

\end{proof}

\subsection{} As we explained, since the analogue of Proposition \ref{prop:pairing} holds, all the remaining contents of Section \ref{sec:proof} apply for $\CS$ replaced by $_\BK\CS$. We thus conclude the following. \\

\begin{corollary} 
\label{cor:gen}

Under Assumption \begin{otherlanguage*}{russian}Ъ\end{otherlanguage*}, we have ${_\BK\oCS} = {_\BK\CS}$. \\

\end{corollary}

\noindent Note that the notion of standard words need not be the same for $\CS$ as for ${_\BK\CS}$. More specifically, for some word $v$ there might exist a relation:
$$
e_v = \sum_{\text{standard } w > v} \text{coeff} \cdot e_w \in \CS
$$
for some coefficients in $\BF$, which does not specialize to $\BK$ (i.e. these coefficients are rational functions in the formal symbols $q$ and $t_e$, which might have poles when evaluating $q,t_e$ to elements of $\BK$). If this were to happen, then $v$ would not be standard with respect to $\CS$, but it might be standard with respect to $_\BK\CS$. This ``failure of flatness" could a priori result in the algebra $_\BK\CS$ being ``bigger" than the algebra $\CS$. \\

\begin{example} 
\label{ex:final}

Let us consider the case when $Q$ is the quiver with one vertex and $g$ loops; in this case, Assumption \begin{otherlanguage*}{russian}Ъ\end{otherlanguage*} requires $|\rho(q)| < |\rho(t_e)| < 1$ for all $e \in \{1,\dots,g\}$. For example, this is the case in the setting of Remark \ref{rem:g loops}, since the Weil numbers of a smooth curve $X$ over $\mathbb{F}_{q^{-1}}$ have absolute value $q^{-\frac 12}$. The shuffle algebra \eqref{eqn:restricted shuffle} then consists of those symmetric Laurent polynomials $R$ such that:
$$
R(z_1,qz_1,z_3,\dots,z_n) \text{ is divisible by } \prod_{a=3}^n \prod_{e=1}^g \left[ \left(z_a - \frac {qz_1}{t_e} \right) \left(z_a - t_e z_1 \right) \right]
$$
This situation will be considered in \cite{Quiver 3}, where we will give a generators-and-relations presentation of a large part of the Hall algebra of vector bundles on $X$, which includes the spherical subalgebra. \\

\end{example}

\subsection{}
\label{sub:special k-ha}

In the present paper, our main reason for considering the setup of \eqref{eqn:restricted shuffle} is to study $K$-theoretic Hall algebras which are equivariant with respect to a subtorus: 
$$
H \subset T = \BC^* \times \prod_{e \in E} \BC^*
$$
In this case, let $\BK = \text{Frac}(\text{Rep}_H)$ throughout the current Subsection, and let the elements $q,t_e \in \BK$ be the restrictions to $H$ of the homonymous characters of $T$. Consider the localized $K$-theoretic Hall algebra, by analogy with \eqref{eqn:intro k-ha loc}:
$$
K_{H,\loc} = K_H \bigotimes_{\text{Rep}_H} \BK
$$
where $K_H$ is the version of $K$ defined by replacing the subscript $T$ with $H$ in \eqref{eqn:k-ha intro}. There exists an algebra homomorphism analogous to \eqref{eqn:shuf intro}:
\begin{equation}
\label{eqn:special homomorphism}
K_{H,\loc} \rightarrow {_{\BK}\CV}
\end{equation}
We have the following analogue of Proposition \ref{prop:yu}. \\

\begin{proposition}
\label{prop:special yu}

The image of \eqref{eqn:special homomorphism} lies in the shuffle algebra of Definition \ref{def:special shuffle}:
$$
K_{H,\emph{loc}} \xrightarrow{\iota_H} {_{\BK}\CS} \subset {_{\BK}\CV}
$$
In fact, this statement even holds before localization, i.e. with $K_H$ instead of $K_{H,\emph{loc}}$ and the shuffle algebra defined over $\emph{Rep}_H$ instead of over $\BK = \emph{Frac}(\emph{Rep}_H)$. \\

\end{proposition}

\begin{proof} We will prove the required statements in a fixed (but arbitrary) graded piece of the algebras $K_{H,\text{loc}}, {_{\BK}\CS}, {_{\BK}\CV}$. For any $i,j \in I$ and for any $\gamma \in \BK^\times$, let us consider all edges: 
$$
e_1,\dots,e_d = \oij \quad \text{and} \quad e_1',\dots,e'_{d'} = \oji
$$
such that:
$$
t_{e_1} = \dots = t_{e_d} = \frac q{t_{e'_1}} = \dots = \frac q{t_{e'_{d'}}} = \gamma
$$
We will henceforth use the notation in the proof of Proposition \ref{prop:yu}. We need to show that for any $\alpha \in K_{H,\text{loc}}$, the shuffle element $R = \iota_H(\alpha)$ has the property that:
\begin{equation}
\label{eqn:division property}
(z_{jb} - z_{ic} \gamma)^{d+d'} \quad \text{divides} \quad R \Big|_{z_{ia} = q z_{ic}}
\end{equation}
(for any $a\neq c$, such that also $a\neq b \neq c$ if $i=j$). To this end, we apply the same argument as in Proposition \ref{prop:yu}, but replacing the locally closed subset $V_e$ of \eqref{eqn:locally closed} by the locally closed subset $V_\gamma$ of collections of linear maps of the following form:
$$
(\phi_e,\phi_e^*) = \begin{cases} (x_r E_{bc}, y_rE_{ab}) &\text{if }e = e_r \text{ for some }r \in \{1,\dots,d\} \\ (y_{r'}' E_{ab}, x'_{r'} E_{bc}) &\text{if }e = e'_{r'} \text{ for some }r' \in \{1,\dots,d'\} \\ (0,0) &\text{otherwise} \end{cases}
$$
where in the formula above the $x$'s and $y$'s are complex numbers which satisfy:
$$
x_1y_1+\dots+x_dy_d - x_1'y_1' - \dots - x_{d'}'y_{d'}' \neq 0
$$
Thus, the set $V_\gamma$ is the complement of a hypersurface in affine space with coordinates $x_r,y_r, x'_{r'}, y_{r'}'$, and the action $H \times (\text{maximal torus of }G_{\bn}) \curvearrowright V_\gamma$ is such that the $x$'s and $y$'s are rescaled by the characters:
$$
\frac {z_{jb}}{\gamma z_{ic}} \qquad \text{and} \qquad \frac {\gamma z_{ia}}{q z_{jb}}
$$
respectively. Let $\rho : V_\gamma \rightarrow (\text{point})$ denote the usual projection. By analogy with the proof of Proposition \ref{prop:yu}, we will use the fact that $\rho^*(R) = 0$ to obtain \eqref{eqn:division property}. First of all, replacing $V_{\gamma}$ by the closed subset:
$$
\overline{V}_\gamma = \Big\{x_1y_1+\dots+x_dy_d - x_1'y_1' - \dots - x_{d'}'y_{d'}' = 1 \Big\}
$$
has the effect of replacing $R$ by $\overline{R} = R|_{z_{ia} = q z_{ic}}$. However, $\overline{V}_\gamma$ is an affine bundle over projective space with coordinates $x_1,\dots,x_d,x_1',\dots,x'_{d'}$, so its equivariant $K$-theory is isomorphic to that of projective space, namely:
$$
\text{Rep}_H[\dots,z_{ia}^{\pm 1},\dots,z_{jb}^{\pm 1},\dots,z_{ic}^{\pm 1}, \dots]\Big|_{z_{ia} = qz_{ic}} \Big/ (z_{jb} - z_{ic}\gamma)^{d+d'}
$$
(the particular quotient is due to the fact that the coordinates $x_1,\dots,x_d,x_1',\dots,x'_{d'}$ are all rescaled by the equivariant parameter $z_{jb} (\gamma z_{ic})^{-1}$). Then the fact that $\overline{R}$ vanishes when pulled back from a point to $\overline{V}_\gamma$ precisely implies \eqref{eqn:division property}.

\end{proof}

\noindent Exactly like we proved the surjectivity of the map \eqref{eqn:iota 2}, we obtain the following. \\

\begin{corollary}
\label{cor:special surj}

If the characters $q,t_e$ of $H$ satisfy Assumption \begin{otherlanguage*}{russian}Ъ\end{otherlanguage*}, then the map:
$$
K_{H,\emph{loc}} \stackrel{\iota_H}\twoheadrightarrow {_{\BK}\CS} 
$$
is surjective (recall that $\BK = \emph{Frac}(\emph{Rep}_H)$). \\

\end{corollary}

\noindent For example, Assumption \begin{otherlanguage*}{russian}Ъ\end{otherlanguage*} holds for the one-dimensional torus:
\begin{equation}
\label{eqn:example}
H = \BC^* \hookrightarrow \BC^* \times \prod_{e \in E} \BC^* = T, \qquad a \mapsto (a^2,a,\dots,a)
\end{equation}
which corresponds to the situation when $t_e = q^{\frac 12}$ for all edges $e$. However, it does not hold for the trivial torus $H = \{1\}$, in which case one must impose stronger conditions than \eqref{eqn:divisible 2} to define the shuffle algebra ${_{\BK}\CS}$ (these will be studied in \cite{Arbitrary}). \\

\noindent It was shown in \cite[Proposition 2.4.4]{VV} (see also \cite[Proposition 3.9]{Pad 1} for a version of this result in the setting of categories of singularities associated to quivers with potential) that the map $\iota_H$ is injective under the condition that $q \neq 1$, $t_e \neq 1$ and $q/t_e \neq 1$ as characters of $H$, for all edges $e \in E$ \footnote{While this condition is a priori weaker than \cite[(2.44)]{VV}, we note that it ensures the existence of a cocharacter $\theta : \BC^* \rightarrow H$ whose fixed point set in the affine space $T^*Z_{\bn}$ consists of only the origin, thus allowing the proof of \cite[Proposition 2.4.4.]{VV} to run through. We thank Michela Varagnolo and \'Eric Vasserot for pointing out this fact.}. Since the aforementioned condition is weaker than Assumption \begin{otherlanguage*}{russian}Ъ\end{otherlanguage*}, we obtain the following. \\

\begin{corollary}
\label{cor:special iso}

If the characters $q,t_e$ of $H$ satisfy Assumption \begin{otherlanguage*}{russian}Ъ\end{otherlanguage*}, then the map:
$$
K_{H,\emph{loc}} \stackrel{\iota_H}\cong {_{\BK}\CS} 
$$
is an isomorphism, hence $K_{H,\emph{loc}}$ is generated by elements of minimal degree. \\

\end{corollary}

\subsection{}
\label{sub:km quiver}

One of the original motivations for $K$-theoretic Hall algebras and shuffle algebras was the fact that they provide incarnations of the positive halves of quantum loop groups of Kac-Moody type associated to symmetric Cartan matrices. In the present Subsection, we assume that $Q$ has no edge loops or multiple edges, and work over the ground field:
\begin{equation}
\label{eqn:km field}
\BK = \BQ(q^{\frac 12})
\end{equation}
The following is the (by now classical) definition of the positive half of the quantum loop group associated to $Q$. \\

\begin{definition}
\label{def:quantum group}

Consider the algebra:
$$
U_q^+(L\mathfrak{g}_Q) = \BK \Big \langle e_{i,d} \Big \rangle_{i \in I, d \in \BZ} \Big / \text{relations \eqref{eqn:km 1}, \eqref{eqn:km 2}, \eqref{eqn:km 3}}
$$
where we consider the formal series $e_i(x) = \sum_{d \in \BZ} \frac {e_{i,d}}{x^d}$ for all $i \in I$, and require: 
\begin{equation}
\label{eqn:km 1}
e_i(x) e_i(y) (xq - y)  = e_i(y) e_i(x) (x- yq) 
\end{equation}
\begin{equation}
\label{eqn:km 2}
e_i(x) e_j(y) \left(x - y q^{\frac 12}\right)  = e_j(y) e_i(x) \left(x q^{\frac 12} - y \right) 
\end{equation}
and:
\begin{multline}
\label{eqn:km 3}
e_i(x_1)e_i(x_2)e_j(y) - \left(q^{\frac 12} + q^{-\frac 12}\right) e_i(x_2)e_j(y)e_i(x_1) + e_j(y) e_i(x_1) e_i(x_2) + \\ + e_i(x_2)e_i(x_1)e_j(y) - \left(q^{\frac 12} + q^{-\frac 12} \right) e_i(x_1)e_j(y)e_i(x_2) + e_j(y) e_i(x_2) e_i(x_1) = 0
\end{multline}
for all $i \neq j$ in $I$. \\

\end{definition}

\noindent It is easy to note that the assignment $e_{i,d} \mapsto z_{i1}^d$ yields an algebra homomorphism:
\begin{equation}
\label{eqn:upsilon}
U_q^+(L\mathfrak{g}_Q) \xrightarrow{\Upsilon} {_q\CS'}
\end{equation}
where the right hand side is the shuffle algebra defined as in Subsection \ref{sub:r}, but over the ground field $\BK$ and with all $t_e$'s set equal to $q^{\frac 12}$ in the definition of the $\zeta'$ function in \eqref{eqn:def zeta final}. Since Assumption \begin{otherlanguage*}{russian}Ъ\end{otherlanguage*} holds in the case at hand, Corollary \ref{cor:gen} (or more precisely, the version of this Corollary which pertains to the twisted setup of Section \ref{sec:r} instead of the untwisted setup of Section \ref{sec:proof}) implies that $\Upsilon$ is surjective. \\

\begin{proof} \emph{of Theorem \ref{thm:triple iso}:} As shown in \cite[Theorem 5.7]{Quiver 3}, the algebra ${_q\CS'}$ is isomorphic to the quantum group $_q\mathbf{U}^+$, which is defined by generators $\{e_{i,d}\}_{i \in I, d \in \BZ}$ modulo relations \eqref{eqn:km 1}, \eqref{eqn:km 2} and:
$$
\left(x_1 - y q^{\frac 12} \right) e_i(x_1) e_i(x_2) e_j(y) + \left(x_2 q^{\frac 12} - x_1 q^{-\frac 12} \right)  e_i(x_2) e_j(y) e_i(x_1) + 
$$
\begin{equation}
\label{eqn:km 4}
 + \left(y q^{-\frac 12}  - x_2\right) e_j(y) e_i(x_1) e_i(x_2)  = 0
\end{equation}
It suffices to show that relations \eqref{eqn:km 3} and \eqref{eqn:km 4} are equivalent modulo relations \eqref{eqn:km 1}, \eqref{eqn:km 2}. The fact that \eqref{eqn:km 4} implies \eqref{eqn:km 3} was proved in \cite[Example 3.9]{Quiver 3}. As for the opposite implication, let us rewrite \eqref{eqn:km 4} as:
$$
(x_1 -x_2) e_i(x_1) e_i(x_2) e_j(y) + \left(x_2 -  y q^{\frac 12} \right) e_i(x_1) e_i(x_2) e_j(y) +
$$
$$
+ \left(x_2 q^{\frac 12} - x_1 q^{-\frac 12} \right)  e_i(x_2) e_j(y) e_i(x_1) +
$$
$$
+ (x_1 - x_2) e_j(y) e_i(x_1) e_i(x_2)  + \left(y q^{-\frac 12} - x_1\right) e_j(y) e_i(x_1) e_i(x_2)
$$
By applying \eqref{eqn:km 2} to the second and fifth terms above, we obtain:
$$
\Big(x_1 -x_2\Big) \Big(e_i(x_1) e_i(x_2) e_j(y) + e_j(y) e_i(x_1) e_i(x_2) \Big) + 
$$
$$
+ \left(x_2 q^{\frac 12} - x_1 q^{-\frac 12} \right) \Big(e_i(x_1) e_j(y) e_i(x_2) + e_i(x_2) e_j(y) e_i(x_1) \Big) 
$$
If we multiply the expression above by $q^{\frac 12}+q^{-\frac 12}$, we obtain:
$$
\left(x_1q^{\frac 12} -x_2q^{-\frac 12} + x_1q^{-\frac 12} -x_2q^{\frac 12}\right) \Big(e_i(x_1) e_i(x_2) e_j(y) + e_j(y) e_i(x_1) e_i(x_2) \Big) + 
$$
$$
+ \left(x_2 q^{\frac 12} - x_1 q^{-\frac 12} \right) \left(q^{\frac 12}+q^{-\frac 12}\right)\Big(e_i(x_1) e_j(y) e_i(x_2) + e_i(x_2) e_j(y) e_i(x_1) \Big) 
$$
which, by applying \eqref{eqn:km 1} to the first line, becomes $x_1 q^{-\frac 12} - x_2 q^{\frac 12}$ times \eqref{eqn:km 3}.

\end{proof}

\begin{proof} \emph{of Corollary \ref{cor:hopf pairing}:} The algebra $U_q^+(L\mathfrak{g}_Q)$ can be extended and made into a bialgebra by analogy with the constructions in Subsection \ref{sub:cop plus}: 
$$
U_q^{\geq}(L\mathfrak{g}_Q) = U_q^+(L\mathfrak{g}_Q) \bigotimes_{\BK} \BK \left[h^+_{i,d}\right]_{i \in I, d \geq 0}
$$
Similarly, one defines $U_q^-(L\mathfrak{g}_Q) = U_q^+(L\mathfrak{g}_Q)^{\op}$ (with generators denoted by $f_{i,d}$) and extends/endows it with a coproduct by analogy with Subsection \ref{sub:cop minus}:
$$
U_q^{\leq}(L\mathfrak{g}_Q) = U_q^-(L\mathfrak{g}_Q) \bigotimes_{\BK} \BK \left[h^-_{i,d}\right]_{i \in I, d \geq 0}
$$
One can construct a natural pairing:
\begin{equation}
\label{eqn:pairing loop}
U_q^+(L\mathfrak{g}_Q) \otimes U_q^-(L\mathfrak{g}_Q) \xrightarrow{\langle \cdot, \cdot\rangle} \BK
\end{equation}
by requiring that it satisfies relation \eqref{eqn:pair h's}, that:
$$
\Big \langle e_{i,d}, f_{j,k} \Big \rangle = \delta^i_j \delta^0_{d+k}, \qquad \forall \ i,j \in I, d,k \in \BZ
$$
and that its extension to $U_q^\geq(L\mathfrak{g}_Q) \otimes U_q^\leq(L\mathfrak{g}_Q)$ is a bialgebra pairing (by analogy with Proposition \ref{prop:bialgebra pairing}). A priori, the pairing \eqref{eqn:pairing loop} might be degenerate, and in fact this would be the case if $Q$ had multiple edges. However, in the situation at hand, the fact that $\Upsilon$ (as well as its analogue when $+$ is replaced with $-$) is an isomorphism which preserves all coproducts and pairings means that \eqref{eqn:pairing loop} coincides with (the version with $\zeta \leadsto \zeta'$ of) the pairing \eqref{eqn:pairing}, which we know to be non-degenerate. 

\end{proof}

\end{document}